\documentclass[12pt]{amsart}

\usepackage[latin1]{inputenc}
\usepackage[english]{babel}
\usepackage{amsthm}
\usepackage{amsmath}
\usepackage{amsxtra}
\usepackage{mathrsfs}
\usepackage{amssymb}
\usepackage{amscd}
\usepackage[pdftex]{graphicx}
\usepackage{graphicx}
\usepackage{stmaryrd}
\usepackage{paralist}
\usepackage{comment}
\usepackage{docmute}
\usepackage{accents}
\usepackage{here}
\usepackage{float}
\usepackage{xcolor}
\usepackage{mathtools}

\newtheorem{theorem}{Theorem}[section]
\newtheorem{lemma}[theorem]{Lemma}
\newtheorem{fact}[theorem]{Fact}
\newtheorem*{claim*}{Claim}

\newtheorem*{theorem*}{Theorem}
\newtheorem*{subclaim*}{Subclaim}
\newtheorem{remark}[theorem]{Remark}
\newtheorem{cor}[theorem]{Corollary}
\newtheorem{question}[theorem]{Question}

\theoremstyle{definition}
\newtheorem{definition}[theorem]{Definition}

\newcommand{\supp}{\rm{supp}}

\newcommand{\mbbQ}{\mathbb{Q}}
\newcommand{\mbbR}{\mathbb{R}}
\newcommand{\mbbP}{\mathbb{P}}

\newcommand{\mbbB}{\mathbb{B}}
\newcommand{\mcalP}{\mathcal{P}}
\newcommand{\mcalI}{\mathcal{I}}
\newcommand{\mcalJ}{\mathcal{J}}

\newcommand{\mcalA}{\mathcal{A}}
\newcommand{\mcalD}{\mathcal{D}}

\newcommand{\mcalE}{\mathcal{E}}

\newcommand{\collkappa}{Coll(\omega,<\kappa)}

\newcommand{\restrict}{{\upharpoonright}}
\newcommand{\bool}[1]{\llbracket {#1} \rrbracket}
\newcommand{\PI}{\mathbb{P}_{\mathcal{I}}}
\newcommand{\PJ}{\mathbb{P}_{\dot{\mathcal{J}}}}
\newcommand{\PImu}
{\mathbb{P}_{\mathcal{I}\upharpoonright\mu}}

\newcommand{\PJmu}
{\mathbb{P}_{\dot{\mathcal{J}}\upharpoonright\mu}}
\newcommand{\urltilde}{\kern -.15em\lower .7ex\hbox{~}\kern .04em}

\begin{document}

\title{Stationary tower forcings and universally Baire sets}
\author{Toshimasa Tanno}
\date{\today}
\address{
	\newline
	Graduate School of System Informatics, Kobe University, \newline
	Rokko-dai 1-1, Nada, Kobe 657-8501,
	Japan
}
\email{211x503x@stu.kobe-u.ac.jp}

\maketitle
\begin{abstract}{
We investigate properties of stationary tower forcings and give conditions on stationary towers to derive the universally Baireness of sets of reals in $L(\mbbR)$.
}
\end{abstract}


\section{Introduction}
 Woodin formulated a large cardinal, which is called Woodin cardinal, and proved the following theorem with Shelah.

\begin{theorem}[Shelah-Woodin  \cite{shelah1990large}]\label{regularity}
 If there are $\omega$ Woodin cardinals and a measurable cardinal above them,
 then every set of reals in $L(\mbbR)$ has the regularity properties: the Lebesgue measurability, the Baire property, and the perfect set property.
\end{theorem}

The universally Baireness is introduced in Feng-Magidor-Woodin \cite{feng1992universally}.
 This property is considered as a generalization of the Baire property,
 and implies the regularity properties (under some large cardinal axioms).

\begin{theorem}[Woodin \cite{neeman1998proper}]\label{univBaire}
 If $\lambda$ is a Woodin cardinal which is a limit of Woodin cardinals,
 then every set of reals in $L(\mbbR)$ is $\lambda$-universally Baire.
\end{theorem}

In the proof of Theorem \ref{regularity} and \ref{univBaire}, large cardinal properties of small uncountable cardinals play important roles.
We can prove Theorem \ref{regularity} using large cardinal properties of a tower consisting of non-stationary ideals $NS_{\omega_1,\alpha}$ on $\mcalP_{\omega_1}(V_\alpha)$ for ordinals $\alpha\geq \omega_1$.
This tower of ideals is called a stationary tower.
 About Theorem \ref{regularity}, Matsubara and Usuba investigated the conditions on the stationary tower which are sufficient to derive the consequence without explicit use of strong large cardinals, such as Woodin cardinals.

 \begin{theorem}[Matsubara-Usuba \cite{matsubara2018countable}]
     If $\lambda$ is a weakly compact cardinal and $\mcalI^{\omega_1}_\lambda=\langle NS_{\omega_1,\alpha}\mid \alpha\in[\omega_1,\lambda)\rangle$ is precipitous,
     then every set of reals in $L(\mbbR)$ has the regularity properties. 
 \end{theorem}

 In this paper, we give a condition on the stationary towers to derive the consequence of Theorem \ref{univBaire}.
 We prove the following.

  \begin{theorem}\label{maintheorem}
  Let $\lambda$ be an inaccessible cardinal.
  Suppose that for unboundedly many successor cardinals $\kappa<\lambda$,
  $\mcalI^\kappa_\lambda=\langle NS_{\kappa,\alpha}\mid \alpha \in[\kappa,\lambda)\rangle$ is presaturated 
  and 
  there are unboundedly many inaccessible cardinals $\mu<\lambda$ such that $\mcalI^\kappa_\mu$ is presaturated.
  Then every set of reals in $L(\mbbR)$ is $\lambda$-universally Baire.
 \end{theorem}

This paper is organized as follows.
 
 In section 2,
 we do preliminaries about forcings and universally Baire sets.

 In section 3,
 we introduce a property of tower forcings, the factorization, and show its relation with the presaturation.
 The factorization of towers is crucial to analyze the reals added by forcings according to its tower.

 In section 4,
 we show that if the height of a tower is a Woodin cardinal then its tower is presaturated.
 In particular, the assumption of Theorem \ref{univBaire} implies the assumption of Theorem \ref{maintheorem}.

 In section 5, 
 we see the presaturation and factorization under L\'{e}vy collapses.
 This plays a significant role in the proof of the main lemma, Lemma \ref{absolute}

 In section 6, 
 we prove the main lemma \ref{absolute} and the main theorem \ref{maintheorem}.

\section{Preliminaries}
In this section, we present some notation and basic facts. For those which are not mentioned here,
consult Jech\cite{jech2003set}, Kunen\cite{kunen2014set} or Kanamori\cite{kanamori2008higher}.
In this paper, we identify $\mbbR$ with $\omega^\omega$.

Let $X$ be a non-empty set and $I\subseteq \mcalP(X)$ be an ideal on $X$.
$I$ is said to be \textit{fine} if $\{a\in X\mid i\notin a\}\in I$ for all $i\in\bigcup X$. $I$ is said to be \textit{normal} if for all function $f\colon X\rightarrow\bigcup X$ such that $\{a\in X\mid f(a)\in a\}\notin I$ there is $b\in\bigcup X$ such that 
    $\{a\in X \mid f(a)=b\}\notin I$.
Unless explicitly stated, we shall restrict our attention to fine and normal ideals.
We let $I^+=\{A\subseteq X\mid A\notin I\}$, the set of $I$-positive sets.

We introduce some notations about preorders and forcings.
Let $\mbbP$, $\mbbQ$ be preorders.
For $p, q\in\mbbP$, $p \parallel q$ means that $p$ and $q$ are compatible, that is, there is $r\in\mbbP$ such that $r\leq_\mbbP p, q$.

For a complete Boolean algebra $\mbbB$ we define $\mbbB^+ =\mbbB\setminus \{0\}$.
For every preorder $\mbbP$ there exists a complete Boolean algebra $\mbbB$ and a dense embedding $i\colon \mbbP\rightarrow\mbbB^+$.
Such $\mbbB$ is called the completion of $\mbbP$.

We denote $\mbbP\simeq\mbbQ$ if the completions of $\mbbP$ and $\mbbQ$ are isomorphic as Boolean algebras.
Then we say $\mbbP$ and $\mbbQ$ are forcing equivalent.

         For preorder $\mbbP$ and $\mathcal{D}\subseteq \mbbP$ we let $\mathcal{D}\downarrow =\{p\in\mbbP\mid \exists q\in\mathcal{D}(p\leq_\mbbP q)\}$.

We recall the definition and basic properties of the L\'{e}vy collapse.

\begin{definition}
 Let $\alpha< \beta$ be ordinals.
     $Coll(\alpha,\beta)$ is the poset of all partial functions $p$ from $\alpha$ to $\beta$ such that $|p|<\alpha$, ordered by reverse inclusions.
     Similarly, $Coll(\alpha, <\beta)$ is the poset of all partial functions $p$ from $\alpha\times\beta$ to $\beta$ such that $|p|<|\alpha|$ and $p(\gamma,\delta)\in\delta$ for all $(\gamma,\delta)\in dom(p)$,
     ordered by reverse inclusions.
     For $\alpha<\beta<\gamma$, 
     \begin{center}
     $Coll(\alpha,[\beta,\gamma))=\{p\in Coll(\alpha,<\gamma)\mid \mathrm{dom}(p)\subseteq \alpha\times[\beta,\gamma)\}$.
     \end{center}
\end{definition}

     The proof of Fact \ref{absorb} and Fact \ref{collapsefactor} can be found in the appendix of Larson \cite{larson}. 
     \begin{fact}\label{absorb}
         If $\mbbP$ is a poset such that $\Vdash_\mbbP |\mbbP|=\aleph_0$,
         then $\mbbP$ has a dense set isomorphic to $Coll(\omega,|\mbbP|).$
         In particular, 
         $\mbbP\times \collkappa \simeq \collkappa$ for $\kappa >|\mbbP|$.
     \end{fact}

     \begin{fact}\label{collapsefactor}
         Let $\lambda$ be an inaccessible cardinal,
         $\mbbP\in V_\lambda$ be a poset, and
         $G$ be a $Coll(\omega,<\lambda)$-generic filter over $V$.
         Suppose that $g\in V[G]$ is a $\mbbP$-generic filter over $V$.
         Then, there is $Coll(\omega,<\lambda)$-generic filter $H \in V[G]$ over $V[g]$ such that $V[G]=V[g][H]$.
     \end{fact}

We give the definition of the universally Baire sets and see basic properties.

\begin{definition}
    Let $\lambda$ be an infinite cardinal.
    $A\subseteq \mbbR$ is $\lambda$-universally Baire if there are sets $X$ and $Y$ and trees $S$ on $\omega\times X$ and $T$ on $\omega\times Y$ such that 
    \begin{itemize}
        \item $p[S]=A$ and $p[T]=\mbbR \setminus A$,
        \item $\Vdash_{\mbbP} `` p[S]\cup p[T]=\mbbR"$ for all poset $\mbbP$ such that $|\mbbP|<\lambda$.
    \end{itemize}
    $A$ is universally Baire if $A$ is $\lambda$-universally Baire for all cardinal $\lambda$.
\end{definition}
The universally Baireness was introduced by Feng-Magidor-Woodin \cite{feng1992universally} as a generalized Baire property and derives regularity properties.

\begin{fact}[Feng-Magidor-Woodin \cite{feng1992universally}]
Let $A\subseteq \mbbR$. The following are equivalent.
\begin{enumerate}
    \item  $A$ is universally Baire.
    \item  $f^{-1}[A]$ has the Baire property for all topological space $X$ and continuous function $f \colon X \rightarrow \mbbR$.
\end{enumerate}
\end{fact}

\begin{fact}[Feng-Magidor-Woodin \cite{feng1992universally}]
Let $A\subseteq \mbbR$.
\begin{enumerate}
    \item If $A$ is universally Baire, then $A$ has the Baire property and is Lebesgue measurable.
    \item If $\lambda$ is a Woodin cardinal and $A$ is $\lambda$-universally Baire, then $A$ has the perfect set property.
\end{enumerate}
\end{fact}

By a basic argument of forcings, we can replace posets $\mbbP$ in the definition of the universally Baireness with L\'{e}vy collapses.

\begin{lemma}
    Let $\lambda$ be an infinite cardinal and $A\subseteq \mbbR$.
    The following are equivalent.
    \begin{enumerate}
             \item $A$ is $\lambda$-universally Baire.
             \item there are sets $X$ and $Y$ and trees $S$ on $\omega\times X$ and $T$ on $\omega\times Y$ such that 
     \begin{itemize}
         \item $p[S]=A, p[T]=\mbbR \setminus A$.
         \item for all $\mu<\lambda$ there exists $\kappa\in [\mu, \lambda)$ such that 
         \begin{center}
         $\Vdash_{\collkappa} ``p[S]\cup p[T] = \mbbR".$
         \end{center}
     \end{itemize}
         \end{enumerate}
\end{lemma}
 \begin{proof}
      It suffices to prove that (2) implies (1).
        Let $S$ and $T$ be trees satisfying conditions in $(2)$ and $\mathbb{P}$ be an arbitrary poset of cardinality less than $\lambda$.
        We show that $V[G]\models ``p[S]\cup p[T]=\mbbR$" for all $\mbbP$-generic filter $G.$
        Let $\kappa\in(|\mbbP|,\lambda)$ be such as in (2).

        By Fact~\ref{absorb},
        $V[G][H]\models ``p[S]\cup p[T]=\mbbR"$
        for all $\collkappa$-generic filter $H$ over $V[G]$.
        Thus for any $r\in \mbbR^{V[G]}$, $r\in p[S]$ or $r\in p[T]$ holds in $V[G][H].$
        If $r\in p[S]$ in $V[G][H],$ then tree 
        \begin{center}
        $S_r=\bigcup_{n\in\omega}\{s\mid \langle r{\upharpoonright}n, s\rangle\in S\}$
        \end{center}
        is ill-founded in $V[G][H].$
        By the absoluteness, $S_r$ is also ill-founded in $V[G].$
        Therefore $V[G]\models r\in p[S].$
        By the same argument, we can show that if $r\in p[T]$ in $V[G][H]$ then $r\in p[T]$ in $V[G].$
     \end{proof}

\section{Tower forcing}

In this section, we first present basic definitions and properties about tower forcings.
Then we define the presaturation and factorization of towers, and show that their properties are related.

 \begin{definition}
    Let $\pi : A \rightarrow B$ be a function and $I$ be an ideal on $A$. Then the \textit{projection} of $I$ to $B$ by $\pi$ is the ideal $I'$ on $B$ defined by 
    \begin{center}
    $X \in I'$ $\Leftrightarrow$ $\pi^{-1} (X) \in I$.
    \end{center}
    
    Let $\langle U, <_U \rangle$ be a linearly ordered set, $\langle W_a \mid a\in U \rangle$ be a sequence of sets with commuting mappings $\pi_{a',a} \colon W_{a'} \rightarrow W_a$ for $a <_U a'$,
    and $I_a$ be an ideal on $W_a$.
    $\mcalI=\langle I_a\mid a\in U\rangle$ is a \textit{tower} if $I_a$ is the projection of $I_{a'}$ to $X_a$ by $\pi_{a,a'}$  for all $a <_U a'$.

    The order type of $U$ is called the \textit{height} of the tower.
    A tower is called $\kappa$-\textit{tower} if each ideal $I_\alpha$ is fine, normal, and  $\kappa$-complete.
 \end{definition}

\begin{remark}
 The projection of fine,  normal and $\kappa$-complete ideal is also fine,  normal and $\kappa$-complete.
\end{remark}

In this paper, we assume that the height of a tower is inaccessible and 
the base sets of ideals in a tower $\mcalI$ of height $\lambda$ have cardinalities less than $\lambda$.
Indeed, specific towers in this paper satisfy these conditions.

We can define a preorder associated
to a tower of ideals.

\begin{definition}\label{defoftowerforcing}
    Let $\mathcal{I} = \langle I_a \mid a\in U \rangle$ be a tower of ideals on sets $\langle W_a \mid a\in U\rangle$ with commuting mappings $\langle \pi_{a',a} \mid a <_U a' \rangle.$
    We define a poset $\mbbP_\mcalI$ on $\bigcup_{a\in U} I^+_a$ as follows.
    For $X, Y\in \mathbb{P}_\mathcal{I}$, fix $a, b\in U$ such that $X\in I^+_a$ and $Y\in I^+_b$.
    We say 
    \begin{center}
    $X \leq_{\mathbb{P}_\mathcal{I}} Y \Leftrightarrow \pi^{-1}_{c, a}(X) \setminus \pi^{-1}_{c, b}(Y) \in I_c$
    \end{center}
    for $c\in U$ such that $a, b \leq_U c.$

    We call this $\mathbb{P}_\mathcal{I}$ the \textit{tower forcing} by $\mathcal{I}.$
    For $X\in\PI$ if there is the minimum $a\in U$ such that $X\in I^+_a$ then we call that $a$ the \textit{support} of $X$ and let denote $\mathrm{supp}(X)$.
    In this paper, we deal towers with wellordered $U$.
\end{definition}

\begin{remark}\label{liftup}

    For $X\in I^+_a$ and $a<_U a'$,
    $X\leq_{\PI} \pi_{a',a}^{-1}(X)$ and $\pi_{a',a}^{-1}(X) \leq_{\PI} X$.
    Thus the order in $\PI$ is independent of the choice of $c$ in Definition \ref{defoftowerforcing}.
\end{remark}

For a $\mathbb{P}_\mathcal{I}$-generic filter over $V$, we can construct the generic ultrapower $Ult(V;G)$ and the generic ultrapower embedding $j \colon V \rightarrow M$.
\begin{lemma}
    Suppose that $U$ is wellordered, $\mathcal{I}=\langle I_a\mid a\in U\rangle$ is a tower on $\langle W_a \mid a\in U\rangle$,
    and $G$ is a $\mathbb{P}_\mathcal{I}$-generic filter over $V$.
    Then, $G^a =\{X^a \mid
    X\in G,
    a\leq_U \mathrm{supp}(X)\}$ is a $V$-ultrafilter over $W_a$,
    where $X^a=\{\pi_{\mathrm{supp}(X),a} (x)\mid x\in X\}$.
\end{lemma}

For a $\mathbb{P}_\mathcal{I}$-generic filter $G$ and $a\in U$, we can construct the generic ultrapower $Ult(V;G^a).$
Let $(M_a; E_a)= Ult(V;G^a)$ and 
$j_a \colon V \rightarrow (M_a; E_a)$
 be the ultrapower embedding.
 For $a<_U a'$
 we define $j_{a,a'} : M_a\rightarrow M_{a'}$ be $j_{a,a'}([f]_{G_a})=[f']_{G_{a'}}$
 where $f$ is a function from $W_a$ to $V$ and $f'$ is the function from $W_{a'}$ to $V$ defined by $f'(x)=f(\pi_{a',a}(x))$.

 Then $j_{a,a'}$ is an elementary embedding,
 and 
 \begin{center}
 $\langle M_a, j_a, j_{a,a'}\mid a,a'\in U, a<_U a' \rangle$
 \end{center}
 is a commuting system.
 Hence we obtain the ultrapower $(M; E)$ and $j : V \rightarrow M$ via $\mcalI$, as the direct limit of this system.

 We give an important example of a tower, the stationary tower.

 \begin{definition}
     Let $\kappa$ be a regular cardinal, $W$ be a set of cardinality at least $\kappa$,
     and $X\subseteq \mcalP_\kappa(W)$.
     \begin{itemize}
         \item $X$ is \textit{closed} if for all $\gamma<\kappa$ and $\subseteq$-increasing sequence $\langle z_\alpha\mid \alpha<\gamma\rangle$ in $X$,
          $\bigcup_{\alpha<\gamma}z_\alpha \in X$.
         \item $X$ is \textit{unbounded} if for all $x\in \mcalP_\kappa(W)$ there is $z\in X$ such that $x\subseteq z$.
         \item $X$ is \textit{club} if $X$ is closed and unbounded.
         \item $X$ is \textit{stationary} if $X\cap C\neq \emptyset$ for every club set $C\subseteq \mcalP_\kappa(W)$.
     \end{itemize}
     We let $NS_{\kappa, \alpha}$ denote the set of all non-stationary sets in $\mcalP_\kappa(V_\alpha)$ for ordinals $\kappa\leq\alpha$.
 \end{definition}

The proof of Fact \ref{stationary} can be found in Foreman-Magidor-Shelah \cite{foreman1988martin}.
\begin{fact}\label{stationary}
    Let $\kappa$ be a regular cardinal and $W$ be a set of cardinality at least $\kappa$.
    For every club set $C\subseteq \mcalP_\kappa(W)$ there is a function $F : [W]^{<\omega}\rightarrow W$ such that 
    \begin{center}
    \{$x\in \mcalP_\kappa(W)\mid x\cap\kappa \in\kappa\land F[[x]^{<\omega}]\subseteq x\} \subseteq C$.
    \end{center}
\end{fact}

The proof of Fact \ref{NStower} can be found in Jech \cite{jech2003set}.
 \begin{fact}\label{NStower}
 Let $\lambda$ be an inaccessible cardinal, $\kappa$ be an uncountable regular cardinal and 
 $\alpha$ be an ordinal with
 $\kappa\leq\alpha<\lambda$.
    Then $NS_{\kappa,\alpha}$ is a fine, normal and $\kappa$-complete ideal on $\mcalP_{\kappa}(V_\alpha)$ and
     $\langle NS_{\kappa,\alpha}\mid \alpha\in[\kappa,\lambda)\rangle$ is a tower with mappings $\pi_{\beta,\alpha}(X)=X\cap V_\alpha$ for $\alpha<\beta<\lambda$.
 \end{fact}
 \begin{definition}
     We call the tower  $\langle NS_{\kappa,\alpha}\mid \alpha\in[\kappa,\lambda)\rangle$ the \textit{stationary tower}
     and denote by $\mcalI^\kappa_\lambda$,
     and let $\mathbb{Q}^\kappa_\lambda$ denote the associated poset to $\mcalI^\kappa_\lambda$.
 \end{definition}

 \begin{definition}
     Let $\kappa$ be a regular cardinal,  $\lambda$ be an inaccessible cardinal with $\kappa<\lambda$ and $\mcalI$ be a $\kappa$-tower on
     $\langle \mcalP_\kappa(V_\alpha)\mid \alpha\in[\kappa,\lambda)\rangle$.
     Suppose that $\alpha<\lambda$ is an ordinal, $A\subseteq V_\alpha$ and $\langle X_a\mid a\in A\rangle$ be a sequence of conditions of $\PI$.
     Let $\alpha^*=\min\{\alpha<\lambda\mid A\subseteq V_\alpha\}$.
     For each $a\in A$
     we let $\eta_\alpha$ be the minimum ordinal such that $\bigcup X_\alpha =V_{\eta_\alpha}$ and $\eta=\max\{\alpha^*, \sup_{a\in A} \eta_a\}$.
     We define 
     \begin{center}
         $\bigtriangledown_{a\in A} X_a =\{x\in \mcalP_\kappa (V_\eta)\mid \exists a \in A\cap x~(x\cap \bigcup X_a \in X_a)\}$
     \end{center} and 
     \begin{center}
         $\bigtriangleup_{a\in A} X_a =\{x\in \mcalP_\kappa (V_\eta)\mid \forall a\in A\cap x~(x\cap \bigcup X_a \in X_a)\}$.
     \end{center}
 \end{definition}
 \begin{lemma}\label{supinf}
     Let $\kappa$ be a regular cardinal,  $\lambda$ be an inaccessible cardinal with $\kappa<\lambda$ and $\mcalI$ be a $\kappa$-tower on
     $\langle \mcalP_\kappa(V_\alpha)\mid \alpha\in[\kappa,\lambda)\rangle$.
     Suppose that $A\subseteq V_\lambda$ and $|A|<\lambda$.
         Then
         $\bigtriangledown_{a\in A} X_a$
         is the supremum of 
         $\{X_a\mid a\in A\}$ in $\PI$, and
         $\bigtriangleup_{a\in A} X_a$
         is the infimum of $\{ X_a\mid a\in A\}$ in $\PI$.
 \end{lemma}

 \begin{proof}
     We only prove the former statement.
     The proof of the latter is essentially the same. 
     To see that 
     $\bigtriangledown_{a\in A} X_a$ is an upper bound in $\PI$,
     it suffices to show that 
     \begin{center}
     $\{x\in \mcalP_\kappa(V_\eta)\mid x\cap \bigcup X_b\in X_b\}\setminus \bigtriangledown_{a\in A} X_a$
     \end{center}
     is non-stationary set for all $b\in A$.
     If $x$ is an element of this set, then $b\notin x$ by the definition of $\bigtriangledown_{a\in A} X_a$. That is, this set does not intersect the club set $\{x\in\mcalP_\kappa(V_\eta)\mid b\in x\}$.

     To see the minimality, we show that if $X_a\leq_{\PI} Y$ for all $a\in A$ then $\bigtriangledown_{a\in A} X_a \leq_{\PI} Y$.
     Since Remark \ref{liftup} and $\bigcup X_a\subseteq V_\eta$
     we may assume that $\bigcup Y =V_\eta.$
     Suppose that $\bigtriangledown_{a\in A} X_a \setminus Y$ is stationary.
     Let $f : \bigtriangledown_{a\in A} X_a \setminus Y \rightarrow A$ be a function such that
     $f(x)\in  A\cap x$.
     By the normality, there exists $b_0\in A$ such that 
     \begin{center}
     \{$x\in \bigtriangledown_{a\in A} X_a\setminus Y\mid x\cap \bigcup X_{b_0}\in X_{b_0} $\}
     \end{center}
     is stationary. It follows that $X_{b_0}\nleq_{\PI} Y$, which is a contradiction.
 \end{proof}
 We consider the presaturation of towers, which guarantee significant properties of induced embedding.
 \begin{definition}
     Let $\lambda$ be an inaccessible cardinal, and $\mcalI$ be a tower of height $\lambda$.
     Then $\mcalI$ is said to be \textit{presaturated} if $\Vdash_{\mbbP_\mcalI} ``\lambda$ is a regular cardinal."
 \end{definition}

 \begin{fact}\label{wellfdd}
     Let $\lambda$ be an inaccessible cardinal and
     $\mcalI$ be a tower of height $\lambda$.
     Suppose that $G$ is a $\mbbP_\mcalI$-generic filter and $(M;E)$ is the ultrapower of $V$ by $G$.
     If $\mcalI$ is presaturated,
     then $(M;E)$ is closed under sequences of length less than $\lambda$ in $V[G]$. 
     In particular, $(M;E)$ is well-founded.
 \end{fact}
 Henceforth, if the ultrapower $(M;E)$ is well-founded then we identify it with its transitive collapse.

  \begin{fact}\label{presat}
     Let $\kappa$ be a successor cardinal and $\lambda$ be an inaccessible cardinal with $\kappa<\lambda$.
     Suppose that $\mcalI$ is a presaturated $\kappa$-tower of ideals on $\langle \mcalP_\kappa(W_\alpha)\mid \alpha\in[\kappa,\lambda)\rangle$ with maps $\pi_{\beta,\alpha}(X)=X\cap W_\alpha$ 
     for $\kappa\leq\alpha<\beta<\lambda$,
     where $W_\alpha$ is a set with $\alpha\subseteq W_\alpha$.
     Let $G$ be a $\mbbP_{\mcalI}$-generic filter over $V$ and $j : V\rightarrow M$ be the ultrapower embedding associated to $G$.
     Then crit$(j)=\kappa$ and $j(\kappa)=\lambda$.
 \end{fact}

 The proof of Fact \ref{wellfdd} and Fact \ref{presat} can be found in Foreman \cite{foreman2010ideals}.

 In the situation which we consider,
 the presaturation is characterized by the chain condition.

 \begin{definition}
 Let $\kappa$ and $\lambda$ be regular cardinals.
 A poset $\mbbP$ is said to have the \textit{weak} $(\kappa,\lambda)$-\textit{c.c.}
 if for all $\mu<\kappa$, sequence
 $\langle \mcalA_{\xi}\mid \xi < \mu\rangle$
 of antichains
 in $\mbbP$
 and $p\in \mbbP$
 there exists $q\leq_\mbbP p$ such that
 $|\{r\in \mcalA_\xi \mid r \parallel q\}|<\lambda$
 for all $\xi<\mu$.
 \end{definition}

 \begin{lemma}\label{chaincondition}
     Let $\lambda$ be an inaccessible cardinal,
     $\mu<\kappa=\mu^+ <\lambda$ be regular cardinals
     and
     $\mbbP$ be a poset of cardinality $\lambda$.
     Suppose that
     $\Vdash_\mbbP |\alpha|\leq \mu$ for all $\alpha\in[\kappa,\lambda)$.
     Then the following are equivalent.
     \begin{enumerate}
         \item $\Vdash_\mbbP ``\lambda$ is a regular cardinal".
         \item $\mbbP$ has the weak $(\kappa, \lambda)$-c.c.
         \item  Let $\langle \mcalA_\xi\mid \xi<\lambda\rangle$ be a sequence of maximal antichains  in $\mbbP$
         and $\{p^\xi_\eta \mid \eta<\lambda_\xi \}$ be an enumeration of $\mcalA_\xi$.
         Then for all $p\in\mbbP$ and club subsets $C\subseteq \lambda$
         there exists $\zeta\in C$ and $q\leq_\mbbP p$ such that 
         $\{\eta<\lambda_\xi \mid q \parallel p^\xi_\eta\}\subseteq \zeta$ for all $\xi <\zeta$.
     \end{enumerate}
 \end{lemma}

 \begin{proof}

  $(3)\Rightarrow (2)$ : Let $p\in\mbbP$ and $\langle \mcalA_\xi\mid \xi<\mu \rangle$ be a sequence of antichains in $\mbbP$. We can take a sequence
  $\langle \mcalA'_\xi \mid \xi<\lambda\rangle$
  of maximal antichains  in $\mbbP$ such that 
  if $\xi<\mu$ then $\mcalA_\xi \subseteq \mcalA'_\xi$.
  By the assumption, there exist $\zeta<\lambda$ and $q\leq_\mbbP p$ such that
  \begin{center}
  $|\{r\in \mcalA_\xi\mid r \parallel q\}|\leq |\{r\in \mcalA'_\xi\mid r \parallel q\}|
  \leq|\zeta|$
  \end{center}
  for all $\xi<\mu$.

  \vskip\baselineskip
  
  $(2)\Rightarrow(1)$ : Let $p\in \mbbP$ and  $\dot{f}$ be a $\mbbP$-name for a function with $p \Vdash_\mbbP \dot{f} : \mu \rightarrow \lambda$.
  It suffices to prove that there is $q\leq_\mbbP p$ which forces that ran$(\dot{f})$ is bounded in $\lambda$.
  For $\xi<\mu$, let $\mcalA_\xi$ be an antichain below $p$ deciding the value of $\dot{f}(\xi)$.
  That is, for all $p'\in \mcalA_\xi$ there is $\eta<\lambda$ such that $p'\Vdash_{\mbbP} \dot{f}(\xi)=\eta$.
  For $p'\in \mcalA_\xi$, let $\eta^\xi_{p'}$ be $\eta<\lambda$ such that $p'\Vdash_{\mbbP} \dot{f}(\xi)=\eta$.
  By the weak ($\kappa,\lambda$)-c.c.,  there exists $q\leq_\mbbP p$ such that
  $|\{r\in \mcalA_\xi \mid r \parallel q\}|<\lambda$
  for all $\xi<\mu$.
  Then $\lambda_\xi =\sup \{\eta^\xi_{p'} \mid q \parallel p'\}$ is less than $\lambda$, and
  \begin{center}
  $q\Vdash_{\mbbP} \mathrm{ran}(\dot{f})\subseteq \bigcup_{\xi<\mu} \lambda_{\xi}$.
  \end{center}
  By the regularity of $\lambda$,
  we have $\bigcup_{\xi<\mu} \lambda_{\xi}<\lambda$.

  \vskip\baselineskip

  $(1)\Rightarrow(3)$ : Let $\langle \mcalA_\xi\mid \xi<\lambda\rangle$ be a sequence of maximal antichains  in $\mbbP$,
  $\lambda_\xi$ be the cardinarity of $\mcalA_\xi$,
  and $\{p^\xi_\eta \mid \eta<\lambda_\xi \}$ be an enumeration of $\mcalA_\xi$.
  Suppose that $p\in\mbbP$ and $C \subseteq \lambda$ is a club subset.
  Let $\dot{f}$ be a $\mbbP$-name of a function such that $\Vdash_\mbbP \dot{f} : \lambda \rightarrow \lambda$ and $p^\xi_\eta \Vdash \dot{f}(\xi)=\eta$. 
  Since $\Vdash_\mbbP ``\lambda$ is a regular cardinal and $C$ is club in $\lambda$",
  \begin{center}
  $\Vdash_\mbbP$ `` $\exists \zeta\in C$ $(\dot{f}[\zeta]\subseteq \zeta)$".
  \end{center}
  Hence there is $q\leq_\mbbP p$ and $\zeta\in C$ such that $q\Vdash_\mbbP \dot{f}[\zeta]\subseteq \zeta$.
  To see that $\{\eta<\lambda_\xi\mid q \parallel p^\xi_\eta\}\subseteq \zeta$ for all $\xi<\zeta$, we fix $\xi<\zeta$ and $\eta<\lambda_\xi$ such that $q \parallel p^\xi_\eta$.
  Since $p^\xi_\eta \Vdash_{\mbbP} \dot{f}(\xi) =\eta$ and $p^\xi_\eta \nVdash_{\mbbP} \dot{f}[\zeta]\nsubseteq \zeta$,
  it follows that $\eta<\zeta$.
 \end{proof}

 Next we introduce the factorization of towers, which plays important roles in the proof of Theorem \ref{univBaire} and Theorem \ref{maintheorem}.
 \begin{definition}
     Let $\mu<\lambda$ be inaccessible cardinals and $\mcalI$ be a tower of height $\lambda$.
     We say that $\mcalI$ \textit{factors} at $\mu$ if for all $X\in \mbbP_{\mcalI\restrict\mu}$ there is $Y \leq_{\mbbP_\mcalI}X$
     in $\mbbP_\mcalI$ such that $Y\Vdash_{\mbbP_\mcalI} ``\dot{G}\cap \mbbP_{\mcalI\restrict\mu}$ is $\mbbP_{\mcalI\restrict\mu}$-generic", where $\dot{G}$ is the canonical name of a $\mbbP_\mcalI$-generic filter.
 \end{definition}

The proof of Fact \ref{symext} can be found in Larson \cite{larson}.

 \begin{fact}\label{symext}
    Let $\lambda$ be an inaccessible cardinal, $\mbbP$ be a poset, and $G$ be a $\mbbP$-generic filter over $V$.
    Suppose that
          \begin{enumerate}
              \item $\lambda=\sup\{\omega_1^{V[r]}\mid r\in \mbbR^{V[G]}\}$.
              \item for all $r\in \mbbR^{V[G]}$ there is $\mbbP_r\in V_\lambda$ such that $r$ is $\mbbP_r$-generic over $V$.
          \end{enumerate}
    Then, in forcing extensions of $V[G]$ in which $(2^\lambda)^{V[G]}$ is countable, there is a $Coll(\omega, <\lambda)$-generic filter $g$ over $V$ such that $\mbbR^{V[G]} = \mbbR^{V[g]}$.
 \end{fact}

 \begin{cor}\label{addedreal}
     Let $\lambda$ be an inaccessible cardinal and
     $\mcalI$ be a tower of height $\lambda$ such that $\mbbP_{\mcalI\restrict\mu}\in V_\mu$ for all $\mu<\lambda$.
     Suppose that there are unboundedly many inaccessible cardinals $\mu<\lambda$ at which $\mcalI$ factors.
     Let $G$ be a $\PI$-generic filter over $V$
     and $j \colon V\rightarrow M$ be the ultrapower embedding by  $G$. 
     Suppose that $\mcalI$ is presaturated, $crit(j)=\omega_1$ and $j(\omega_1)=\lambda$.
     Then, in forcing extensions of $V[G]$ in which $(2^\lambda)^{V[G]}$ is countable,
     there is a $Coll(\omega,<\lambda)$-generic filter $g$ over $V$ such that $\mbbR^{V[G]}=\mbbR^{V[g]}$.
 \end{cor}
 \begin{proof}
     We see that two conditions in Fact \ref{symext} are satisfied.
     We note that $\mbbR^{V[G]}=\mbbR^M$ 
     for the ultrapower $M$ by the presaturation and Fact \ref{wellfdd}.

     1. $[\leq]$ Let $\alpha<\lambda$. Since $M\models \lambda=\omega_1$,
     there is $r\in \mbbR^M$ which codes a bijection between $\omega$ and $\alpha$. For such $r$, $\alpha\leq{\omega_1}^{V[r]}$.

     $[\geq]$
     For $r\in\mbbR^{V[G]}$, ${\omega_1}^{V[r]} \leq {\omega_1}^M =\lambda$.
     
     2. For $\alpha\in[\omega_1,\lambda)$ let 
     \begin{center}
     $\mathcal{E}_\alpha =\{X\in \mbbP_{\mcalI} \mid X\Vdash_{\mbbP_{\mcalI}} $``$\exists\mu\in(\alpha,\lambda)$ 
      $(\mu$ is an inaccessible cardinal and
      
      $\dot{G}\cap \mbbP_{\mcalI\restrict \mu}$
     is $\PImu$-generic)"\}.
     \end{center}
     By the assumption, $\mathcal{E}_\alpha$ is dense for all $\alpha\in[\omega_1,\lambda)$, so the genericity implies that 
     \begin{center}
     $E=\{\mu<\lambda\mid \mu$ is inaccessible and $G\cap \PImu$ is $\PImu$-generic\}
     \end{center}
     is unbounded.
     Let $r\in \mbbR^{V[G]}$ and $\dot{r}$ be a $\PI$-name for $r$.
     For each $n<\omega$, there is a maximal antichain $\mcalA_n$ in $\PI$ deciding the value of $\dot{r}(n)$.
     Since $\lambda={\omega_1}^{V[G]}$,
     there is $\mu\in E$ such that $G\cap \mcalA_n \cap V_\mu \neq\emptyset$ for all $n<\omega$.
     Then $r\in V[G\cap\PImu]$, so $r$ is $\PImu$-generic over $V$.
 \end{proof}

The rest of this section is devoted to the relation between the presaturation and the factorization.

 \begin{lemma}\label{lambda}
    Let $\kappa$ be a successor cardinal and $\lambda$ be an inaccessible cardinal with $\kappa<\lambda$.
    If $\mcalI^\kappa_\lambda$ is presaturated, then
    for all $X\in \mbbQ^\kappa_\lambda$ and
    sequence 
    $\langle \mcalA_\xi\mid \xi<\lambda\rangle$
    of maximal antichains in $\mbbQ^\kappa_\lambda$
    there are stationary many $x\in\mcalP_\kappa(V_\lambda)$ such that 
    \begin{enumerate}
        \item $x\cap\bigcup X\in X$,
        \item for all $\xi\in\lambda\cap x$ there is $Y\in \mcalA_\xi \cap x$ such that $x\cap\bigcup Y\in Y$.
    \end{enumerate}
 \end{lemma}
 \begin{proof}
    Let $X\in\mbbQ^\kappa_\lambda$ and $\langle \mcalA_\xi\mid \xi<\lambda\rangle$ be a sequence of maximal antichains in $\mbbQ^\kappa_\lambda$. 
    We show that for all function $f : [V_\lambda]^{<\omega}\rightarrow V_\lambda$ there is $x\in \mcalP_\kappa(V_\lambda)$ such that $x$ satisfies (1), (2), $x\cap \kappa\in\kappa$ and $f[{[x]}^{<\omega}]\subseteq x$.
    We enumerate $\mcalA_\xi=\{ X^\xi_\eta\mid \eta<\lambda_\xi\}$ for $\xi<\lambda$. 
    By (3) of Lemma \ref{chaincondition},
    there are $Y\leq_{\mbbQ^\kappa_\lambda} X$ and $\zeta<\lambda$ which satisfy the following conditions.
    \begin{enumerate}[(i)]
        \item $\{\eta<\lambda_\xi\mid X^\xi_\eta \parallel Y\}\subseteq \zeta$ for all $\xi<\zeta$,
        \item $\bigcup X^\xi_\eta \subseteq V_\zeta$ for all $\xi, \eta<\zeta$,
        \item $f[[V_\zeta]^{<\omega}]\subseteq V_\zeta$,
        \item $\bigcup X \subseteq V_\zeta$.
    \end{enumerate}
    We note that the set of $\zeta<\lambda$ satisfying $(\mathrm{ii}),(\mathrm{iii})$ and $(\mathrm{iv})$ is club in $\lambda$.
    Since $\mcalA_\xi$ is a maximal antichain, it follows that 
    \begin{center}
    $Y\leq \sup_{\eta<\zeta} X^\xi_\eta$ 
    \end{center}
    for all $\xi<\zeta$ from the first condition.
    Hence we obtain that 
    \begin{center}
    $\inf_{\xi<\zeta} \sup_{\eta<\zeta} X^\xi_\eta= \bigtriangleup_{\xi<\zeta} \bigtriangledown_{\eta<\zeta} X^\xi_\eta$
    \end{center}
    is stationary and compatible with $X$ in $\mbbQ^\kappa_\lambda$.
    So we can take $x\in \mcalP_\kappa(V_\zeta)$ such that 
    \begin{itemize}
        \item $x\cap\kappa\in\kappa$,
        \item $f[[x]^{<\omega}]\subseteq x$,
        \item $x\cap\bigcup X\in X$,
        \item for all $\xi\in \zeta\cap x$ there exists $\eta\in \zeta\cap x$ such that $x\cap\bigcup X^\xi_\eta \in X^\xi_\eta$,
        \item if $\eta, \xi \in x$ then $X^\xi_\eta \in x$.
    \end{itemize}
    Then this $x$ is as desired.
 \end{proof}

 \begin{lemma}\label{lambdaplus}
    Let $\kappa$ be a regular cardinal and $\lambda$ be an inaccessible cardinal with $\kappa<\lambda$.
     Suppose that for all $X\in \mbbQ^\kappa_\lambda$ and
    sequence of maximal antichains $\langle \mcalA_\xi\mid \xi<\lambda\rangle$ in $\mbbQ^\kappa_\lambda$,
    there are stationary many $x\in \mcalP_{\kappa}(V_{\lambda})$ such that
    \begin{enumerate}
        \item $x\cap \bigcup X \in X$,
        \item for all $\xi\in\lambda\cap x$ there is $Y\in \mcalA_{\xi} \cap x$ such that $x\cap \bigcup Y \in Y$.
    \end{enumerate}
    Then 
    there are stationary many $x\in \mcalP_\kappa(V_{\lambda+1})$ such that 
    \begin{enumerate}[(1')]
        \item $x\cap \bigcup X\in X$,
        \item for all maximal antichain $\mcalA\in x$ in $\mbbP$ there is $Y\in \mcalA\cap x$ with $x \cap \bigcup Y\in Y$.
    \end{enumerate}
 \end{lemma}
 \begin{proof}
    We show that for all $f : [V_{\lambda+1}]^{<\omega} \rightarrow V_{\lambda+1}$ there is  $x\in \mcalP_\kappa(V_{\lambda+1})$ such that $x$ satisfies (1'), (2'), $x\cap\kappa\in\kappa$ and $f[[x]^{<\omega}]\subseteq x$.
    Take a set $M$ such that $V_\lambda\subseteq M\subseteq V_{\lambda+1}$, $|M|=\lambda$,
    and $f[[M]^{<\omega}]\subseteq M$.
    We enumerate maximal antichains in $\mbbQ^\kappa_\lambda$ in $M$ as $\langle \mcalA_\xi \mid \xi<\lambda\rangle$.
    By the assumption,
    there are stationary many $x\in \mcalP_\kappa(M)$ such that
    \begin{itemize}
        \item $x\cap \bigcup X\in X$,
        \item for all $\xi\in\lambda\cap x$ there is $Y\in\mcalA_\xi \cap x$ such that $x\cap\bigcup Y\in Y$. 
    \end{itemize}
     So we can take $x\in \mcalP_\kappa(M)$ such that
    \begin{itemize}
    \item $x\cap\kappa\in\kappa$, 
        \item $f[[x]^{<\omega}]\subseteq x$,
        \item $x\cap\bigcup X\in X$,
        \item for all $\xi\in \lambda\cap x$ there exists $Y\in \mcalA_\xi\cap x$ such that $x\cap\bigcup Y \in Y$,
        \item if $\mcalA_\xi\in x$ then $\xi\in x$.
    \end{itemize}

    Since $x\subseteq M$, for all maximal antichain $\mcalA$ in $x$ there is $\xi <\lambda$ such that $\mcalA=\mcalA_\xi$, and then $\xi \in x$. Thus, there is $Y\in \mcalA_\xi\cap x$ such that $x\cap\bigcup Y\in Y$ by (2) of the assumption.
    That is, $x$ satisfies the condition $(2')$. \end{proof}

 \begin{lemma}\label{factoratmu}
     Let $\lambda$ be an inaccessible cardinal,
     $\kappa\leq\mu<\lambda$ be regular cardinals,
     and
     $Y$ is the set of $x\in\mcalP_\kappa(V_{\mu+1})$ such that for all maximal antichains $\mcalA\in x$ in $\mbbQ^\kappa_\mu$ there is $X\in \mcalA\cap x$ such that $x\cap\bigcup X\in X$.
     Suppose that
     $Y$
     is stationary.
     Then, 
     $Y\Vdash_{\mbbQ^\kappa_\lambda} ``\dot{G}\cap \mbbQ^\kappa_\lambda$ is $\mbbQ^\kappa_\mu$-generic",
     where $\dot{G}$ is the canonical name for a $\mbbQ^\kappa_\lambda$-generic filter.
 \end{lemma}
\begin{proof}
    Let $G$ be a $\mbbQ^\kappa_\lambda$-generic filter over $V$ with $Y\in G$ and $\mcalA$ be a maximal antichain in $\mbbQ^\kappa_\mu$.
    We show that $G\cap \mcalA\neq \emptyset$.
    Let $Y_{\mcalA}=\{x\in a\mid \mcalA\in x\}$.
    Since a club set $\{x\in\mcalP_\kappa(V_{\mu+1})\mid \mcalA\in x\}$ is in $G$,
    $Y_\mcalA \in G$.
    We define a function $f \colon Y_\mcalA\rightarrow A$ by $f(x)=X$ such that $X\in \mcalA\cap x$ and $x\cap\bigcup X\in X$.
    By the normality and the genericity, we obtain $X\in \mcalA$ such that 
    \begin{center}
     $Y_{\mcalA,X}=\{x\in Y_\mcalA\mid f(x)=X\}\in G$.
    \end{center}
    Then $X\in G$, as $Y_{\mcalA,X}$.
\end{proof}

From previous lemmas, we obtain the following theorem.

 \begin{theorem}\label{imply}
     Let $\kappa$ be a successor cardinal
     and
     $\mu$, $\lambda$ be inaccessible cardinals with $\kappa<\mu<\lambda$.
     If $\mcalI^\kappa_ \mu$ is presaturated, then $\mathcal{I}^\kappa_\lambda$ factors at $\mu$.
 \end{theorem}

 \begin{proof}
     Applying Lemma \ref{lambda} and \ref{lambdaplus} with $\lambda$ replaced by $\mu$,
     we obtain the stationary set in the assumption of Lemma \ref{factoratmu}.
 \end{proof}

 \section{Woodin cardinals imply the presaturation}
 In this section, we show that the assumption of Theorem \ref{univBaire} implies the assumption of Theorem \ref{maintheorem}.
 Indeed, we prove that if $\lambda$ is a Woodin cardinal and $\kappa<\lambda$ is a successor cardinal then $\mcalI^\kappa_\lambda$ is presaturated.

 \begin{definition}
     Let $M$ and $N$ be sets.
     $N$ \textit{end-extends} $M$ if $M\subseteq N$ and there is an ordinal $\beta$ such that $M\subseteq V_\beta$ and $N\cap V_\beta =M$.
 \end{definition}

\begin{definition}
    Let $\kappa$ be a regular cardinal, $\lambda$ be an inaccessible cardinal with $\kappa<\lambda$ and
    $\mcalD\subseteq \mbbQ^\kappa_\lambda$.
     Let $\mathsf{sp}^\kappa_\lambda (\mcalD)$ be the set of all $M\in \mcalP_\kappa(V_{\lambda+1})$ such that $\mcalD\in M\prec V_{\lambda+1}$ and there is $N\in\mcalP_\kappa(V_{\lambda+1})$ with the following properties.

     \begin{itemize}
         \item $M\subseteq N\prec V_{\lambda+1}$,
         \item $N\cap V_\lambda$ end-extends $M\cap V_\lambda$,
         \item there is $X\in \mcalD\cap N$ such that $N\cap\bigcup X\in X$.
     \end{itemize}
     We say $\mcalD\subseteq \mbbQ^\kappa_\lambda$ is \textit{semi-proper} if $\mathsf{sp}^\kappa_\lambda (\mcalD)$ contains a club set in $\mcalP_\kappa(V_{\lambda+1})$.
\end{definition}

Fact \ref{woodincard} in the case that $\kappa=\omega_1$ is Lemma 2.5.6. in \cite{larson},
and the same proof works for regular cardinals in general.

\begin{fact}\label{woodincard}
    Let $\kappa$ be a regular cardinal, $\lambda$ be a Woodin cardinal with $\kappa<\lambda$
    and $\langle \mcalA_\alpha\mid \alpha<\lambda\rangle$ be a sequence of maximal antichains in $\mbbQ^\kappa_\lambda$.
    Then there is an inaccessible cardinal $\gamma<\lambda$ such that $\mcalA_\alpha \cap \mbbQ^\kappa_\gamma$ is a maximal antichain in $\mbbQ^\kappa_\gamma$ and semi-proper for all $\alpha<\gamma$.
    Moreover, such inaccessible cardinals exist unboundedly in $\lambda$.
\end{fact}

Using this fact, we show the presaturation of $\mcalI^\kappa_\lambda$.

 \begin{theorem}\label{woodinimpliescc}
     Let $\kappa$ be a regular cardinal and $\lambda$ be a Woodin cardinal with $\kappa<\lambda$.
     Then $\mcalI^\kappa_\lambda
     $ is presaturated.
 \end{theorem}

 \begin{proof}
  By Lemma \ref{chaincondition},
  it suffices to prove that $\mbbQ^\kappa_\lambda$ has the weak $(\kappa,\lambda)$-c.c..
  Let $\mu<\kappa$ be a cardinal, $\langle \mcalA_i \mid i<\mu\rangle$ be a sequence of maximal antichains in $\mbbQ^\kappa_\lambda$ and $X\in \mbbQ^\kappa_\lambda$. We can take $\gamma>\mathrm{supp}(X)$ such as in Fact \ref{woodincard}.
\vskip0.5\baselineskip
  \textbf{Claim 1.}
  Let $N\prec V_\lambda$ with $\langle \mcalA_\alpha\cap \mbbQ^\kappa_\gamma\mid\alpha<\mu\rangle, \gamma\in N, |N|<\kappa$ and $N\cap\kappa\in\kappa$.
  Then, there is $N' \in\mcalP_\kappa(V_\lambda)$ such that 
  \begin{itemize}
      \item $N'\prec V_\lambda$,
      \item $N'\cap V_\gamma$ end-extends $N\cap V_\gamma$,
      \item for all $\alpha<\mu$ there is $Y\in\mcalA_\alpha\cap  \mbbQ^\kappa_\gamma \cap N'$ such that $N'\cap \bigcup Y\in Y$.
  \end{itemize}
\vskip0.5\baselineskip
  \textbf{Proof of Claim 1.}
  It suffices to construct $\langle N_\alpha\mid \alpha<\mu\rangle$ such that 
  \begin{itemize}
      \item $|N_\alpha|<\kappa$,
      \item $N\subseteq N_\alpha\prec V_\lambda$,
      \item if $\alpha<\beta<\mu$ then $N_\beta\cap V_\gamma$ end-extends $N_\alpha\cap V_\gamma$,
      \item there is $Y\in \mcalA_\alpha\cap\mbbQ^\kappa_\gamma \cap N_{\alpha+1}$ such that $N_{\alpha+1}\cap \bigcup Y\in Y$.
  \end{itemize}
  If there are such $N_\alpha$, 
  then $N'=\bigcup_{\alpha<\mu} N_\alpha$ is as desired.
  Note that $N'\cap V_\gamma$ end-extends $N_\alpha \cap V_\gamma$, so if $\alpha<\mu$ and $Y\in \mcalA_\alpha\cap \mbbQ^\kappa_\gamma \cap N_{\alpha+1}$ such that $N_{\alpha+1}\cap \bigcup Y\in Y$ then $N'\cap\bigcup Y =
  N_\alpha \cap\bigcup Y\in Y$.

   Let $N_0=N$.

   When $\langle N_\beta\mid \beta<\alpha\rangle$ have been constructed for a limit ordinal $\alpha$, we define $N_\alpha=\bigcup_{\beta<\alpha} N_\beta$.

  Suppose that $\langle N_\beta\mid\beta\leq\alpha\rangle$ has been constructed for $\alpha<\mu$.
  We define $N_{\alpha+1}$.
  For the construction, we claim that 
  there is $N^\ast\in \mcalP_\kappa(V_{\gamma+1})$ such that 
  \begin{itemize}
      \item $N^\ast\prec V_{\gamma+1}$,
      \item $N_\alpha\cap V_{\gamma+1} \subseteq N^\ast$,
      \item $N^\ast\cap V_\gamma$ end-extends $N_\alpha\cap V_\gamma$,
      \item there is $Y\in \mcalA_{\alpha}\cap\mbbQ^\kappa_\gamma\cap N^\ast$ such that $N^\ast\cap\bigcup Y \in Y$.
  \end{itemize}
  To see this, note that $N_\alpha \cap \kappa\in\kappa$.
  Since $\mcalA_\alpha\cap \mbbQ^\kappa_\gamma\in N_\alpha\prec V_\lambda$,
  $N_\alpha \cap V_{\gamma+1}\in \mathsf{sp}^\kappa_\gamma (\mcalA_\alpha \cap\mbbQ^\kappa_\gamma)$.
  So required $N^*$ exists by the definition of $\mathsf{sp}^\kappa_\lambda$.

  Let 
  \begin{center}
  $N_{\alpha+1}=\{f(s)\mid f \colon V_\gamma\rightarrow V_\lambda,f\in N_\alpha,  
  s\in N^\ast \cap V_\gamma\}$.
  \end{center}
  Then $N_{\alpha+1}\prec 
 V_\lambda$.
  Considering the identity and constant functions, we get $N^\ast \cap V_\gamma\subseteq N_{\alpha+1}$ and $N_\alpha\subseteq N_{\alpha+1}$. 
  
  We see that $N_{\alpha+1} \cap V_\gamma$ end-extends $N_\alpha\cap V_\gamma$.
   It suffices to prove that $N_{\alpha+1}\cap V_\gamma = N^\ast\cap V_\gamma$.
   To prove $N_{\alpha+1} \cap V_\gamma\subseteq N^\ast\cap V_\gamma$,
   let $f \colon V_\gamma\rightarrow V_\lambda$,
   $f\in N_\alpha$, $s\in N^\ast\cap V_\gamma$ and $f(s)\in V_\gamma$.
   We define a function $f' \colon V_\gamma\rightarrow V_\gamma$ by $f'(t)=f(t)$ if $f(t)\in V_\gamma$ and $f'(t)=0$ otherwise.
   Then $f'\in N_\alpha\cap V_{\gamma+1}$ and $f'(s)=f(s)$.
   Moreover, $f'(s)\in N^*\cap V_{\gamma}$ since $f',s\in N^*\prec V_{\gamma+1}$.
   So $f(s)\in N^\ast\cap V_\gamma$.
   
   Since $N_{\alpha+1}\cap V_\gamma = N^* \cap V_\gamma$, 
   \begin{center}
   $N_{\alpha+1}\cap \bigcup Y= N^* \cap \bigcup Y$ 
   \end{center}
   for $Y\in \mbbQ^\kappa_\gamma$, and
   $N_{\alpha+1}\cap \mbbQ^\kappa_\gamma=N^*\cap \mbbQ^\kappa_\gamma$.
   Therefore $N_{\alpha+1}$ satisfies conditions.
   
   The proof of Claim 1 is completed.  \hfill$\blacksquare$
\vskip0.5\baselineskip
   \textbf{Claim 2.}
   Let $Z$ be  the set of $M\in \mcalP_{\kappa}(V_\gamma)$ such that 
   \begin{itemize}
       \item $M\cap \bigcup X\in X$,
       \item for all $\alpha<\mu$ there is $Y\in \mcalA_{\alpha}\cap \mbbQ^\kappa_\lambda\cap M$ such that $M\cap \bigcup Y\in Y$.
   \end{itemize}
   Then $Z$ is stationary.

\vskip0.5\baselineskip
   \textbf{Proof of Claim 2.}
   Let $f \colon [V_\gamma]^{<\omega}\rightarrow V_\gamma$ be a function.
   We show that there is $M\in Z$ such that $f[[M]^{<\omega}]\subseteq M$ and $M\cap\kappa\in\kappa$.
   Since $X$ is stationary, there is $N\in\mcalP_\kappa(V_{\lambda})$ such that 
   \begin{itemize}
       \item $X, f, \langle \mcalA_\alpha\mid \alpha<\mu\rangle, \gamma\in N\prec V_{\lambda}$,
       \item $N\cap \bigcup X\in X$,
       \item $N\cap\kappa\in\kappa$.
   \end{itemize}
   For this $N$, we take $N'$ such as in Claim 1 and let $M=N'\cap V_\gamma$.
   Since $M=N'\cap V_\gamma$ end-extends $N\cap V_\gamma$,
   \begin{center}
   $M\cap\bigcup X=N\cap \bigcup X\in X$.
   \end{center}
   If $\alpha<\mu$ then there is $Y\in \mcalA_\alpha\cap \mbbQ^\kappa_\gamma \cap N'$ such that $N'\cap \bigcup Y\in Y$ and such $Y$ is in $M$.
   Thus $M\in Z$.
   And $f[[M]^{<\omega}]\subseteq M$ holds, as $f\in M\prec V_{\gamma}$.
   Since $M$ end-extends $N\cap V_\gamma$ and $N\cap\kappa\in\kappa$, $M\cap\kappa\in\kappa$.
   The proof of Claim 2 is completed.
   \hfill$\blacksquare$
\vskip0.5\baselineskip
   We note that $Z\leq_{\mbbQ^\kappa_\lambda} X$ and 
   \begin{center}
   $Z\leq_{\mbbQ^\kappa_\lambda} \bigtriangledown_{Y\in \mcalA_\alpha \cap\mbbQ^\kappa_\gamma} (\pi^{-1}_{\gamma, \supp(Y)}(Y))$.
   \end{center}
   for all $\alpha<\mu$.
   By Lemma \ref{supinf},
   \begin{center}
   $Z\leq_{\mbbQ^\kappa_\lambda} \sup \mcalA_\alpha \cap\mbbQ^\kappa_\gamma$
   \end{center}
   for all $\alpha<\mu$.
   Since $\mcalA_\alpha$ is a maximal antichain in $\mbbQ^\kappa_\lambda$,
   for all $\alpha<\mu$
   \begin{center}
   $\{Y\in \mathcal{A}_\alpha\mid Y \parallel Z\}\subseteq
   \mcalA_\alpha\cap\mbbQ^\kappa_\gamma$
   \end{center}
   whose cardinality is less than $\lambda$.
 \end{proof}

\begin{cor}
 Let $\lambda$ be a Woodin cardinal which is limit of Woodin cardinals.
 Then 
 for all successor cardinal $\kappa<\lambda$
 there are unboundedly many inaccessible cardinals $\mu<\lambda$ such that $\mcalI^\kappa_\mu$ is presaturated. 
\end{cor}

 \section{Preservation}
We see the preservation of the presaturation and the  factorization under small L\`{e}vy collapses.
Roughly speaking, in order to derive the universally Baireness these preservation play similar roles to the preservation of Woodin cardinals under small forcings in the original proof using large cardinals.

 \begin{lemma}\label{kappacc}
     Let $\kappa$ be a regular cardinal and $I$ be a $\kappa$-complete ideal on a set $Z$.
     Suppose that $\mbbP$ is a poset with $\kappa$-c.c., 
     and $G$ is a $\mbbP$-generic filter over $V$. 
     Let 
     \begin{center}
     $J=\{Y\subseteq Z \mid \exists X\in I(Y\subseteq X)\}$
     \end{center}
     in $V[G]$
     and $\dot{J}$ be a $\mbbP$-name for $J$.
     Then the following hold.
     \begin{enumerate}
     \item For all $p\in \mbbP$ and $\mbbP$-name $\dot{Y}$ such that $p\Vdash_\mbbP  \dot{Y}\in \dot{J}$, there is $X\in I$ such that $p\Vdash_\mbbP \dot{Y}\subseteq X$.
     \item $J$ is  a $\kappa$-complete ideal on $Z$ in $V[G]$.
     \end{enumerate}
 \end{lemma}
 
 \begin{proof}
  (1)
  Suppose that $p\in \mbbP$ and $\dot{Y}$ is a $\mbbP$-name such that $p\Vdash_\mbbP \dot{Y}\in \dot{J}$.
  Let $A$ be a maximal antichain in $\mbbP$ below $p$ such that for each $q\in A$ there is $X_q \in I$ such that 
  $q \Vdash_\mbbP \dot{Y}\subseteq X_q$.
  Let $X=\bigcup_{q\in A}X_q$. Then $X\in I$ by the $\kappa$-c.c. and completeness, and $p\Vdash_\mbbP \dot{Y}\subseteq X$.

  (2)
  To see the $\kappa$-completeness, let $\gamma<\kappa$, $p\in \mbbP$ and $\dot{\bar{Y}}$ be a $\mbbP$-name such that $p\Vdash_{\mbbP} `` \dot{\bar{Y}}$ is a sequence in $\dot{J}$ of length $\gamma$ ".
  By $p\Vdash_{\mbbP} \dot{\bar{Y}}(\xi)\in \dot{J}$ and (1),
  for $\xi<\gamma$
  there is $X_\xi \in I$ such that $p\Vdash_{\mbbP} \dot{\bar{Y}}(\xi)\subseteq X_\xi$.
  By the $\kappa$-completeness of $I$,
  we have that
  $\bigcup_{\xi<\gamma} X_\xi\in I$.
  So $p\Vdash_{\mbbP} \bigcup_{\xi<\gamma} \dot{\bar{Y}}(\xi) \in \dot{J}$.
 \end{proof}

 \begin{lemma}\label{generated}
     Let $\mcalI=\langle I_\alpha\mid \alpha\in[\kappa,\lambda) \rangle$ be a $\kappa$-tower on $\langle \mcalP_\kappa (V_\alpha) \mid \alpha\in [\kappa,\lambda) \rangle$.
     Suppose that
     $\mbbP$ is a poset with the $\kappa$-c.c. and $G$ is a $\mbbP$-generic filter over $V$. 
     Let 
     \begin{center}
     $J_\alpha =\{Y\subseteq \mathcal{P}_\kappa(V_\alpha)^V\mid \exists X\in I_\alpha
     (Y\subseteq X)\}$
     \end{center}
     in $V[G]$.
     Then, $\mcalJ=\langle J_\alpha\mid \alpha<\lambda \rangle$ is a $\kappa$-tower on $\langle \mcalP_\kappa(V_\alpha)^V \mid \alpha\in [\kappa,\lambda)\rangle$ in $V[G]$.
 \end{lemma}

 \begin{proof}
  Let $\alpha$ and $\beta$ be ordinals in $V$ such that $\kappa\leq \beta <\alpha<\lambda$.
  It suffices to prove that 
  \begin{center}
  $Y_\beta \in J_\beta \Leftrightarrow Y_\alpha=\{y\in \mcalP_\kappa(V_\alpha)^V\mid y\cap V_\beta \in Y_\beta\}\in J_\alpha$
  \end{center}
  for all $Y_\beta \subseteq \mcalP_\kappa(V_\beta)^V$ in $V[G]$.

  For the forward direction,
  let $\dot{Y}_\beta$ be a $\mbbP$-name for $Y_\beta$, $\dot{J}_\beta$ be a $\mbbP$-name for $J_\beta$ and $p\in G$ such that $p\Vdash_\mbbP \dot{Y}_\beta\in \dot{J}_\beta$.
  From Lemma \ref{kappacc}, there is $X_\beta\in I_\beta$ such that $p\Vdash \dot{Y}_\beta \subseteq X_\beta$.
  Since $\mcalI$ is a tower, 
  \begin{center}
  $X_\alpha=\{x\in \mcalP_\kappa(V_\alpha)^V\mid x\cap V_\beta\in X_\beta\}\in I_\alpha$.
  \end{center}
  Thus $Y_\alpha\subseteq X_\alpha\in I_\alpha$ in $V[G]$.

  For the reverse direction,
  suppose that $Y_\beta \subseteq \mcalP_\kappa(V_\beta)^V$ and
  \begin{center}
  $Y_\alpha=\{y\in \mcalP_\kappa(V_\alpha)^V \mid y\cap V_\beta \in Y_\beta\}\in J_\alpha$
  \end{center}
  in $V[G]$.
  We show that $Y_\beta \in J_\beta$.
  Working in $V$.
  Let  $\dot{Y}_\alpha$ and $\dot{Y}_\beta$ be $\mbbP$-names for $Y_\alpha$ and $Y_\beta$ respectively, and $p\in G$ such that $p\Vdash_\mbbP \dot{Y}_\alpha \in\dot{J}_\alpha$.
  Let 
  \begin{center}
  $X_\beta=\{x\in\mcalP_\kappa(V_\beta)^V\mid \exists q\leq p~(q\Vdash_\mbbP x\in \dot{Y}_\beta)\}$
  \end{center}
  and we show that $X_\beta \in I_\beta$.
  Since $\mcalI$ is a tower, it suffices to prove that 
  \begin{center}
  $\{x\in \mcalP_\kappa(V_\alpha)^V\mid x\cap V_\beta \in X_\beta\}\in I_\alpha$.
  \end{center}
  From Lemma \ref{kappacc}, there is $X_\alpha \in I_\alpha$ such that $p\Vdash_{\mbbP} \dot{Y}_\alpha \subseteq X_\alpha$.
  If $x\cap V_\beta\in X_\beta$, then there is $q\leq p$ such that $q \Vdash_{\mbbP} x\in \dot{Y}_\alpha \subseteq X_\alpha$, so $x\in X_\alpha$.
  This completes the proof, as $\{x\in \mcalP_\kappa(V_\alpha)^V\mid x\cap V_\beta \in X_\beta\} \subseteq X_\alpha \in I_\alpha$.
 \end{proof}
We call a tower $\mcalJ$ in Lemma \ref{generated} a tower \textit{generated} by $\mcalI$ in $V[G]$.

The following lemma is a tower version of a well-known theorem about ideals, Foreman's duality theorem.
The proof of ideal version is essentially found in Lemma 22.32. in \cite{jech2003set}.
 \begin{lemma}\label{duality}
  Let $\kappa$ be a successor cardinal and $\lambda$ be an inaccessible cardinal with $\kappa<\lambda$.
  Suppose that 
  $\mbbQ$ is a complete Boolean algebra with the $\kappa$-c.c.
  and
  $\mcalI$ is a presaturated $\kappa$-tower on $\langle \mcalP_\kappa(V_\alpha)\mid \alpha\in [\kappa,\lambda)\rangle$.
  Let
  $j : V\rightarrow M$ be an associated embedding by $\mathcal{I}$
  and $\dot{\mcalJ}$ be a $\mbbQ^+$-name for the tower generated by $\mcalI$ in $V^{\mbbQ^+}$.
  Then there is a dense embedding from $\mbbQ^+\ast\PJ$ to $\mbbP_\mcalI \ast j(\mbbQ^+)$.
  In particular, $\mbbQ^+\ast\mbbP_{\dot{\mcalJ}} \simeq \mbbP_\mcalI\ast j(\mbbQ^+)$.
 \end{lemma}

 \begin{proof}
 We define a set
 \begin{center}
  $\mcalD=\{(p,\dot{S})\in \mbbQ^+\ast\mbbP_{\dot{\mcalJ}}\mid \exists \mu<\lambda~(p\Vdash_{\mbbQ^+} \mathrm{supp}(\dot{S})=\mu)\}$.
  \end{center}
  Note that for every $(p,\dot{S})\in\mbbQ^+\ast\PJ$ there is a $\mbbQ^+$-name $\dot{S'}$ such that $(p,\dot{S})\leq_{\mbbQ^+\ast\PJ} (p,\dot{S'})\leq_{\mbbQ^+\ast\PJ} (p,\dot{S})$ and $(p,\dot{S'})\in \mcalD$.
  To see this, we take the minimum $\mu<\lambda$ such that $p\Vdash_{\mbbQ^+} \mathrm{supp}(\dot{S})<\mu$ using the $\kappa$-c.c. of $\mbbQ$.
  Let $\dot{S'}$ be a $\mbbQ$-name for $\pi^{-1}_{\mu,\mathrm{supp}(\dot{S})}(\dot{S})$.

  So it suffices to prove that there is a dense embedding from $\mbbQ^+\ast\PJ$ to $\PI\ast j(\mbbQ^+)$. Let $h : \mcalD \rightarrow \mbbP_{\mcalI}\ast j(\mbbQ^+)$ be a function defined by $h(p,\Dot{S})=(Z, [f]_{\Dot{G}})$, 
  where $\dot{G}$ is the canonical $\PI$-name for a generic filter, 
  \begin{center}
  $Z=\{x\in\mcalP_\kappa (V_\mu)^V\mid  \bool{x\in \Dot{S}}_{\mbbQ^+}  \parallel p\}$
  \end{center}
  and $f : \mcalP_\kappa(V_\mu)^V \rightarrow \mbbQ$ by 
  $f(x)=\llbracket x\in \dot{S}\rrbracket_{\mbbQ^+} \land p$
  for $\mu<\lambda$ such that $p\Vdash \mathrm{supp}(\dot{S})=\mu$.
  $[f]_{\dot{G}}$ is a name of the element represented by $f$ in $Ult(V; G)$.
  Note that $Z\in I_\mu^+$ because $p\Vdash_{\mbbQ^+} \dot{S}\subseteq Z$.

  We can check straightforwardly the preservation of the order and incompatibility by $h$.
  To prove that $h$ has the dense image,
  we fix $(Y,\dot{q})\in \mbbP_{\mcalI} \ast j(\mbbQ^+)$.
  Then  there is $Y'\leq_{\PI} Y$, $\mu<\lambda$, and $f : \mcalP_\kappa(V_\mu)^V \rightarrow \mbbQ^+$ such that $Y'\Vdash_{\PI} \dot{q}=[f]_{\dot{G}}$.
  Note that 
  $(Y',[f]_{\dot{G}})\leq_{\mbbP_{\mcalI} \ast j(\mbbQ^+)} (Y,\dot{q})$.
  We claim that there is $p\in \mbbQ^+$ such that 
  \begin{center}
  $p\Vdash_{\mbbQ^+} \{x\in Y'\mid f(x)\in \dot{G}\}\in \dot{J}^+_\mu$
  \end{center}
  For contradiction, suppose that there is no such condition $p\in\mbbQ^+$. Then
  \begin{center}
  $\Vdash_{\mbbQ^+} \{x\in Y'\mid f(x)\in \dot{G}\}\in \dot{J}_\mu$,
  \end{center}
  so there is $X\in I_\mu$ such that 
  \begin{center}
  $\Vdash_{\mbbQ^+} \{x\in Y'\mid f(x)\in \dot{G}\}\subseteq X$
  \end{center}
  by Lemma \ref{kappacc}.
  Since $Y'\in I_\mu^+$ and $X\in I_\mu$, 
  $Y'\setminus X$ is not empty.
  If $x\in Y'\setminus X$ then $f(x)\nVdash_{\mbbQ^+} f(x)\in \dot{G}$; this is a contradiction.
  
  Let $p\in\mbbQ^+$ be such a condition, and $\dot{X}=\{(\check{x},f(x))\mid x\in Y'\}$.
  Then $p\Vdash_{\mbbQ^+} \dot{X}\in \dot{J}_{\mu}^+$.
  We let $h(p,\dot{X})=(Y^* ,[f^*]_{\dot{G}})$,
  that is, 
  \begin{center}
  $f^{*}(x)=\left\{
  \begin{array}{ll}
     f(x)\land p  & (x\in Y')  \\
      0 & (otherwise).
  \end{array}
  \right.$
  \end{center}
  Then $Y^* \subseteq Y'$ 
  and 
  \begin{center}
  $Y^*\Vdash_{\mbbP_{\mcalI}} [f^*]_{\dot{G}} \leq_{j(\mbbQ)} [f]_{\dot{G}}$.
  \end{center}
  So, $h(p,\dot{X})\leq_{\PI*j(\mbbQ^+)} (Y', [f]_{\dot{G}})\leq_{\PI*j(\mbbQ^+)} (Y,\dot{q})$.
  \end{proof}
 
From this lemma, we can show the preservation of the presaturation.
 \begin{theorem}\label{presatpres}
 Let $\kappa$ be a successor cardinal, $\lambda$ be an inaccessible cardinal with $\kappa<\lambda$,
  $\mcalI$ be a $\kappa$-tower on $\langle \mcalP_\kappa(V_\alpha)\mid \alpha\in [\kappa,\lambda)\rangle$,
  and $\dot{\mcalJ}$ be a $\collkappa$-name of the tower generated by $\mcalI$ in $V^{\collkappa}$.
  Suppose that $\mcalI$ is presaturated.
  Then $\Vdash_{\collkappa} ``\dot{\mcalJ}$ is presaturated".
 \end{theorem}
 \begin{proof}
  Let $j:V\rightarrow M$ be the associated embedding by $\mcalI$.
  We show that
  \begin{center}
  $\collkappa\ast \mbbP_{\dot{\mcalJ}} \simeq \mbbP_{\mcalI}\times Coll(\omega,<\lambda)$.
  \end{center}
  Let $\mbbB_\kappa$ be the completion of $\collkappa$, and $d : \collkappa \rightarrow \mbbB_\kappa^+$ be a dense embedding.
  We obtain $e : V^{\collkappa}\rightarrow V^{\mbbB^+_\kappa}$ induced by $d$. That is, 
  \begin{center}
  $e(\dot{x})=\{(e(\dot{y}),d(p))\mid(\dot{y},p)\in \dot{x}\}$.
  \end{center}
  We note that $\Vdash_{\collkappa} ``e(\dot{\mcalJ})$ is the tower generated by $\mcalI$".
  Since $\collkappa$ and $\mbbB_\kappa^+$ are forcing equivalent, $\collkappa\ast\PJ$ and $\mbbB_\kappa^+\ast\mbbP_{e(\dot{\mathcal{J}})}$ are also forcing equivalent.
  By Lemma \ref{duality}, there is a embedding $h : \mbbB_\kappa^+ \ast \PJ\rightarrow \PI\ast j({\mbbB}_\kappa^+)$.
  For a $\PI$-generic filter $G$ over $V$ and the ultrapower $j : V\rightarrow M$,
  \begin{center}
    $M\models `` j(d)$ is a dense embedding from $Coll(\omega,<\lambda)$ to $j({\mbbB}_\kappa^+) "$.
  \end{center}
  by the elementarity of $j$ and Fact \ref{presat}.
  This holds also in $V[G]$,
  so
  \begin{center}
  $\PI\ast j(\mbbB_\kappa^+)\simeq\PI\times Coll(\omega,<\lambda)$.
  \end{center}
  Combining above arguments, we get
  $Coll(\omega,<\kappa)\ast\PJ\simeq \PI\times Coll(\omega,<\lambda)$.
  
  Since $\mbbP_{\mcalI}\times Coll(\omega,<\lambda)$ preserves the regularity of $\lambda$ by the presaturation of $\mcalI$, 
  it follows that
  $\Vdash_{\collkappa} `` \mbbP_{\dot{\mcalJ}}$ preserves the regularity of $\lambda$"
  .
 \end{proof}

Next we shall consider the preservation of the factorization.
We see that the factorization of $\mcalI$ is characterized by the existence of a special condition in $\PI$.

\begin{lemma}
    Let $\mu<\lambda$ be inaccessible cardinals and $\mcalI$ be a tower of height $\lambda$.

    Then $\mcalI$ factors at $\mu$ if and only if 
    there is $X\in \PI$ such that 
    \begin{enumerate}
        \item $X\parallel Y$ in $\PI$ for all $Y\in \PImu$,
        \item 
        $X\Vdash_{\PI} ``\dot{H}\cap \PImu$ is $\PImu$-generic", where $\dot{H}$ is the canonical name of a $\PI$-generic filter.
    \end{enumerate}
\end{lemma}
\begin{proof}
    We show that if $\mcalI$ factors at $\mu$ then there is a condition $X$ such as in the statement.
    For each $Y\in \PImu$ fix $X_Y \in \PI$ such that $X_Y \leq_{\PI} Y$ and $X_Y \Vdash_{\PI} ``\dot{H}\cap \PImu$ is $\PImu$-generic". 
    Let $X=\bigtriangledown_{Y\in\PImu} X_Y$.
    By Lemma \ref{supinf} $X$ is the supremum of $X_Y$'s in $\PI$.
    So $X$ is as desired.
\end{proof}

We show the preservation of the factorization under L\'{e}vy collapses by showing that the properties $(1)$ and $(2)$ are preserved.

\begin{lemma}\label{compatible}
 Let $\kappa$ be a successor cardinal and $\mu$ and $\lambda$ be inaccessible cardinals with $\kappa<\mu<\lambda$.
 Suppose that $\mcalI$ is a $\kappa$-tower on $\langle \mcalP_\kappa (V_\alpha)\mid \alpha\in[\kappa,\lambda)\rangle$.
 Let $\mbbQ$ be a poset with $\kappa$-c.c. and
 $\dot{\mcalJ}$ is a $\mbbQ$-name of the tower generated by $\mcalI$ in $V^{\mbbQ}$.
 Let $X\in \PI$  and assume that $X\parallel Y$ in $\PI$ for all $Y\in \PImu$.
 Then $\Vdash_{\mbbQ}`` X\parallel Y$ in $\PJ$ for all $Y\in \PJmu$".
\end{lemma}
\begin{proof}
    Let $\alpha\in[\mu,\lambda)$ be the support of $X$.
    Fix $p\in\mbbQ$ and $\mbbQ$-name $\dot{Y}$ such that 
    $p\Vdash_{\mbbQ} \dot{Y}\in \PJmu$.
    We show that there is $q\leq_{\mbbQ} p$ such that $q\Vdash_{\mbbQ} ``X\parallel \dot{Y}$ in $\PJ$".
    We fix $p'\leq_{\mbbQ}p$ and $\beta\in[\kappa,\mu)$ such that
    $p'\Vdash_{\mbbQ} \bigcup\dot{Y}=V_\beta$.
    Let
    \begin{center}
        $Y=\{x\in\mcalP_\kappa (V_\beta)^V\mid \exists r\leq_{\mbbQ} p' (r\Vdash_{\mbbQ} x\in \dot{Y})\}$.
    \end{center}
    Note that $p'\Vdash_{\mbbQ} \dot{Y}\subseteq Y$ and $Y\in I^+_\beta$.
    By assumption, $X$ is compatible with $Y$.

    We show that there is $q\leq_{\mbbQ} p'$ such that $q\Vdash_{\mbbQ} ``X\parallel \dot{Y}$ in $\PJ$".
    Assume not. Then $p'$ forces that $X$ and $\dot{Y}$ are incompatible. That is,
    \begin{center}
    $p'\Vdash_{\mbbQ}\{x\in X\mid x\cap V_\beta\in\dot{Y}\}\in J_\alpha$.
    \end{center}
    By Lemma \ref{kappacc} there is $Z\in I_\alpha$ such that
    \begin{center}
    $p'\Vdash_{\mbbQ} \{x\in X\mid x\cap V_\beta\in\dot{Y}\}\subseteq Z$.
    \end{center}
    If $x\in X$ and $x\cap V_\beta \in Y$ then there is $r\leq_{\mbbQ} p'$ such that $r\Vdash_{\mbbQ} x\cap V_\beta \in \dot{Y}$. Thus
    $(r\Vdash_{\mbbQ}) x\in Z$.
    It implies that $\{x\in X\mid x\cap V_\beta \in Y\}\subseteq Z$,
    but it is impossible because the left-hand is in $I_\alpha^+$ and the right-side is in $I_\alpha$.
\end{proof}

\begin{lemma}\label{mastercondition}
    Let $\kappa$ be a successor cardinal and $\mu$ and $\lambda$ be inaccessible cardinals with $\kappa<\mu<\lambda$.
    Suppose that $\mcalI$ is a $\kappa$-tower on $\langle \mcalP_\kappa(V_\alpha)\mid\alpha\in[\kappa,\lambda)\rangle$ and 
    $\dot{\mcalJ}$ is a $\collkappa$-name of the tower generated by $\mcalI$ in $V^{\collkappa}$.
    Suppose that $\mcalI$ and $\mcalI\restrict\mu$ is presaturated.
    Assume that $X\in \mbbP_{\mcalI}$ and $X\parallel Y$ for all $Y\in \PImu$ and $X\Vdash_{\PI} ``\dot{G}\cap \mbbP_{\mcalI\restrict\mu}$ is $\mbbP_{\mcalI\restrict\mu}$-generic",
    where $\dot{G}$ is the canonical name of a  $\PI$-generic filter.
    Then $\Vdash_{\collkappa} `` X\Vdash_{\PJ} \dot{H}\cap \mbbP_{\dot{J}\restrict\mu}$ is $\mbbP_{\dot{J}\restrict\mu}$-generic, where $\dot{H}$ is the canonical name of $\PJ$-generic filter".
\end{lemma}

\begin{proof}
    Let $\mbbB_\kappa$ be a complete Boolean algebra in which $Coll(\omega,<\kappa)$ is dense as a sub algebra.
    Since $Coll(\omega,<\kappa)$ has $\kappa$-c.c., we may take such $\mbbB_\kappa$ with $\mbbB_\kappa \subseteq H_\kappa$.
    Note that we can consider $\dot{\mcalJ}$ as $\mbbB_\kappa^+$-name naturally.

    It suffices to prove that
    \begin{center}
        $\Vdash_{\mbbB_\kappa^+} ``\dot{\mcalE}$ is predense in $\PJ$ below $X$"
    \end{center}
    for all $\mbbB^+_\kappa$-name $\dot{\mcalE}$ such that 
     $\Vdash_{\mbbB^+_\kappa} ``\dot{\mcalE}$ is predense in $\PJmu$ ".

     Through this proof, let $\dot{G}, \dot{H}$ be canonical names of $\PImu, \PI$-generic filters respectively.
     Let $j_\lambda, j_\mu$ be ultrapower maps by $\dot{G}, \dot{H}$ respectively.
     We fix a $\PI$-name $\dot{\mbbB}_\lambda$ of $j_\lambda (\mbbB_\kappa)$ and 
     a $\PImu$-name $\dot{\mbbB}_\mu$ of $j_\mu (\mbbB_\kappa)$.
     Then by elementarity and Lemma \ref{presat},
     \begin{center}
         $\Vdash_{\PI} ``Coll(\omega,<\lambda)$ is dense in $\dot{\mbbB}^+_\lambda$ and $\dot{\mbbB}_\lambda \subseteq H_\lambda$"
     \end{center}
     and
    \begin{center}
         $\Vdash_{\PImu} ``Coll(\omega,<\mu)$ is dense in $\dot{\mbbB}^+_\mu$ and $\dot{\mbbB}_\mu \subseteq H_\mu$".
    \end{center}

    By Lemma \ref{duality} we have dense embeddings
        $h_\lambda\colon \mbbB^+_\kappa\ast \PJ \rightarrow\PI\ast\dot{\mbbB}^+_\lambda$ and
        $h_\mu\colon \mbbB^+_\kappa\ast \PJmu \rightarrow\PImu\ast\dot{\mbbB}^+_\mu$.

    Let $\gamma>\mu$ be the support of $X$.
    For $S,T\in \PI$ we define
    \begin{center}
     $S\land T=\{x\in\mcalP_\kappa (\bigcup S\cup \bigcup T)\mid x\cap\bigcup S\in S \land
     x\cap\bigcup T\in T\}$.
    \end{center}

    We define $k_0\colon \mbbB^+_\kappa\ast\PJmu\rightarrow \mbbB^+_\kappa\ast\PJ$ by $k_0((p,\dot{Z}))=(p, \dot{Z}\land X)$.
    Note that $\Vdash_{\mbbB^+_\kappa} \dot{Z} \parallel X$ for all $\mbbB^+_\kappa$-name $\dot{Z}$ of a condition of $\PJmu$ by Lemma \ref{compatible}. 

    To define $k_1\colon \PImu\ast\dot{\mbbB}^+_\mu\rightarrow
    \PI\ast\dot{\mbbB}^+_\lambda$,
    fix $(Z,\dot{p})\in \PImu\ast\dot{\mbbB}^+_\mu$.
    Since $Z\Vdash_{\PI} \dot{p}\in j_\lambda (\dot{\mbbB}^+_\kappa)$,
    there are $\PImu$-names 
    $\dot{f}$ and $\dot{\eta}$
    such that 
    \begin{center}$Z\Vdash_{\PImu} ``\dot{f}\in V \colon \mcalP_\kappa(V_{\dot{\eta}})\rightarrow\mbbB_\kappa $ and $[\dot{f}]_{\dot{G}} =\dot{p}$ ".
    \end{center}
    We define 
    $k_1 ((Z,\dot{p})=(Z\land X, [\dot{f}]_{\dot{H}})$.
    \vskip0.5\baselineskip
    \begin{center}
    $
    \begin{CD}
        \mathbb{B}^+_\kappa\ast \PJ @>{h_\lambda}>>
        \PI \ast\dot{\mathbb{B}}^+_\lambda \\
        @A{k_0}AA    @A{k_1}AA \\
        \mathbb{B}^+_\kappa\ast \PJmu @>{h_\mu}>>
        \PImu \ast\dot{\mathbb{B}}^+_\mu
    \end{CD}
    $
    \end{center}
    \vskip0.5\baselineskip
    We show that $k_1\circ h_\mu$ and $h_\lambda \circ k_0$ are almost same in the following sense.
    \vskip0.5\baselineskip
    \textbf{Claim 1.}
    $k_1 \circ h_\mu ((p,\dot{Z}))\sim_{\PI\ast\dot{\mbbB}^+_\lambda}h_\lambda\circ k_0((p,\dot{Z}))$ 
      for $(p,\dot{Z})\in \mbbB^+_\kappa \ast \PJmu$,
      where $(p_0, \dot{Z_0})\sim_{\PI\ast\dot{\mbbB}^+_\lambda} (p_1, \dot{Z_1})$ if and only if $(p_0, \dot{Z_0})\leq_{\PI\ast\dot{\mbbB}^+_\lambda}(p_1, \dot{Z_1})$ and $(p_0, \dot{Z_0})\geq_{\PI\ast\dot{\mbbB}^+_\lambda}(p_1, \dot{Z_1})$.
      \vskip0.5\baselineskip
      \textbf{Proof of Claim 1.}
      Let $(p,\dot{Z})\in \mbbB^+_\kappa \ast \PJmu$ and $\eta<\mu$ be the minimum ordinal such that $p\Vdash_{\mbbB^+_\kappa} \mathrm{supp}(\dot{Z})\leq\eta$.
      Let $\dot{Z'}$ be a $\mbbB^+_\kappa$-name of $\pi^{-1}_{\eta,{\mathrm{supp}({\dot{Z}}})}(\dot{Z})$.
      Note that 
      $p\Vdash_{\mbbB^+_\kappa} \mathrm{supp}(\dot{Z'}) = \eta$.
      Let
      \begin{center}
          $S=\{x\in\mcalP_\kappa(V_\eta)\mid\bool{x\in\dot{Z'}}\parallel p\}, \ 
          T=\{x\in\mcalP_\kappa(V_\gamma)\mid\bool{x\in\dot{Z}\land X}\parallel p\}$,
      \end{center}
          $f\colon \mcalP_\kappa(V_\eta)\rightarrow \mbbB_\kappa$ defined by $f(x)=\bool{ x\in \dot{Z'} }\land p$
          and 
          $g\colon \mcalP_\kappa(V_\gamma)\rightarrow \mbbB_\kappa$ defined by $f(x)=\bool{x\in\dot{Z}\land X}\land p$.
          Then by the construction of $h_\mu$ and $h_\lambda$ in the proof of Lemma \ref{duality},
          $h_\mu ((p,\dot{Z}))=(S,[f]_{\dot{G}})$ and 
          $h_\lambda (k_0 ((p,\dot{Z}))) =h_\lambda((p,\dot{Z}\land X))=(T,[g]_{\dot{H}})$.
          Since $p\Vdash_{\mbbB^+_\kappa} \dot{Z}\land X =\dot{Z'}\land X$, it follows that $T=S\land X$ and
          $g(x)=f(x\cap V_\eta)$ for $x\in X$. So, $X\Vdash_{\PI} [g]_{\dot{H}}=[f]_{\dot{H}}$.
          \hfill$\blacksquare$
              \vskip0.5\baselineskip
    \textbf{Claim 2.}
    Let $\dot{\eta}, \dot{f}$ be $\PImu$-names such that $\Vdash_{\PImu} \dot{f}\colon \mcalP_\kappa(V_{\dot{\eta}})\rightarrow \mbbB_\kappa$.
    Then $X\Vdash_{\PI} [\dot{f}]_{\dot{H}}=[\dot{f}]_{\dot{H}\cap\PImu}$.
    \vskip0.5\baselineskip
    \textbf{Proof of Claim 2.}
Let $X\in H$ be a $\PI$-generic filter and $f= \dot{f}^{H\cap \PImu}$.

We show that $[f]_{H}=[f]_{H\cap \PImu}$.
We define 
          an elementary embedding 
          \begin{center}
          $i\colon Ult(V;H\cap \PImu)\rightarrow Ult(V;H)$
          \end{center}
          by $i([g]_{H\cap\PImu})=[g]_H$.
          Note that there is a canonical function which represents $\alpha$ for each $\alpha<\mu$.
          More precisely, let $\chi_\alpha\colon \mcalP_\kappa (V_\eta)\rightarrow\alpha$ be $\chi_\alpha (x)=o.t.(x\cap\alpha)$ then 
          $\Vdash_{\PI} [\chi_\alpha]_{\dot{H}}=\alpha$ and $X\Vdash_{\PI} [\chi_\alpha]_{\dot{H}\cap\PImu}=\alpha$.
          So crit$(i)\geq \mu$ and
          $i\restrict H_\mu^{V[H\cap \PImu]}$ is identity.
          Since $\dot{\mbbB}^{H\cap\PImu}_\mu \subseteq H_\mu^{V[H\cap\PImu]}$ and $[f]_{H\cap\PImu} \in \mbbB_\mu^{V[H\cap\PImu]}$, $[f]_H =
          i([f]_{H\cap\PImu})=[f]_{H\cap\PImu}$ holds.
          \hfill$\blacksquare$
          
          \vskip0.5\baselineskip
          By Claim 2, it follows that 
          if $k_1 ((Z,\dot{p}))=(Z\land X,[\dot{f}]_{\dot{H}})$ then 
          \begin{center}
          $Z\land X\Vdash_{\mbbB^+_\kappa} [\dot{f}]_{\dot{H}}=\dot{p}$.
          \end{center}
          Also, 
          $X\Vdash_{\PI} \dot{\mbbB}_\mu \subseteq \dot{\mbbB}_\lambda$ because $\dot{\mbbB}_\lambda$ and $\dot{\mbbB}_\mu$ are names of $j_\lambda(\mbbB_\kappa)$ and $j_\mu(\mbbB_\kappa)$.

          \vskip0.5\baselineskip
          \textbf{Claim 3.} Let $\mcalE$ be a predense subset of $\PImu\ast \dot{\mbbB}^+_\mu$.
          Then $k_1[\mcalE]$ is predense in $\PI\ast\mbbB^+_\lambda$ below $(X,\dot{1}_{\dot{\mbbB}^+_\lambda})$.
          \vskip0.5\baselineskip
          \textbf{Proof of Claim 3.}
          Let $H$ be an $\PI\ast\dot{\mbbB}^+_\lambda$-generic filter over $V$ with $(X,\dot{1}_{\dot{\mbbB}^+_\lambda})\in H$. 
          It suffice to prove that ${k_1}^{-1}[H]$ is $\PImu\ast\dot{\mbbB}^+_\mu$-generic over $V$.

          Work in $V[H]$.
          Let
          \begin{center}
            $H_0 =\{Y\in\PI\mid (Y,\dot{1}_{\dot{\mbbB}^+_\lambda})\in H\}$, 
              $H_1=\{\dot{p}^{H_0}\in (\dot{\mbbB}^+_\mu)^{H_0}\mid \exists Y\in H_0 \ ((Y,\dot{p})\in H)\}$.
          \end{center}
          Note that  $H=\{(Y,\dot{p})\mid Y\in {H_0}  \land  \dot{p}^{H_0} \in H_1\}$.

          Let $K_0 =H_0\cap \PImu =\{Z\in \PImu\mid Z\land X\in H_0\}$. Then $K_0$ is $\PImu$-generic over $V$.
          Let $K_1 = H_1 \cap (\dot{\mbbB}^+_\mu)^{K_0}$.

          We claim that $K_1$ is $(\dot{\mbbB}^+_\mu)^{K_0}$-generic over $V[K_0]$.
          In $V[H_0]$, $(\dot{\mbbB}^+_\mu)^{K_0} \subseteq (\dot{\mbbB}^+_\lambda)^{H_0}$ as the observation below Claim 2.
          Since $Coll(\omega,<\mu)$ is a complete subalgebra of $Coll(\omega,<\lambda)$, $({\dot{\mbbB}^+_\mu})^{K_0}$ is complete subalgebra of $({\dot{\mbbB}^+_\lambda})^{H_0}$.
          So $K_1$ is $(\dot{\mbbB}^+_\mu)^{K_0}$-generic over $V[H_0]$, in particular over $V[K_0]$.

          Thus $K_0 \ast K_1$ is $\PImu\ast \dot{\mbbB}^+_\mu$-generic over $V$.
          We show that $(Z,\dot{p})\in K_0\ast K_1$ if and only if $k_1 ((Z,\dot{p})) \in H_0\ast H_1$.
          Fix $(Z,\dot{p})\in K_0\ast K_1$ and let $k_1((Z,\dot{p}))=(Z\land X, \dot{p}^*)$.
          Note that $Z\land X \Vdash_{\PI} \dot{p}=\dot{p}^*$.
          \begin{align*}
              (Z,\dot{p})\in K_0\ast K_1 
              &\Leftrightarrow
              Z\land X\in H_0 \land \dot{p}^{K_0}\in H_1\\
              &\Leftrightarrow
              Z\land X\in H_0 \land (\dot{p}^*)^{H_0} \in H_1\\
              &\Leftrightarrow
              (Z\land X, \dot{p}^*)\in H_0\ast H_1.
          \end{align*}
          This completes the proof.
          \hfill$\blacksquare$
          \vskip0.5\baselineskip
          \textbf{Claim 4.}
          Let $\mcalD$ be a predense subset of $\mbbB^+_\kappa\ast \PImu$,
          then $k_0[\mcalD]$ is predense in $\mbbB^+_\kappa\ast\PI$ below $(1_{\mbbB^+_\kappa}, X)$.

          \vskip0.5\baselineskip
          \textbf{Proof of Claim 4.}
          $h_\mu[\mcalD]$ is dense in $\PImu\ast\dot{\mbbB}^+_\mu$, so $k_1[h_\mu[\mcalD]]$ is predense in $\PI\ast\dot{\mbbB}^+_\lambda$ below $(X,1_{\dot{\mbbB}^+_\lambda})$ by Claim 3.
          By Claim 1, $h_\lambda[k_0[\mcalD]]\downarrow =
          k_1[h_\mu[\mcalD]]\downarrow$.
          Thus $h_\lambda[k_0[\mcalD]]$ is predense in $\PI\ast\dot{\mbbB}^+_\lambda$ below $(X,1_{\dot{\mbbB}^+_\lambda})$.
          Since $h_\lambda((1_{\mbbB^+_\kappa\ast\PJ},X))=
          (X,1_{\dot{\mbbB}^+_\lambda})$ by definition of $h_\lambda$,
          $k_0[\mcalD]$ is predense in $\mbbB^+_\kappa\ast \PJ$ below $(1_{\mbbB^+_\kappa}, X)$.
          \hfill$\blacksquare$

          \vskip0.5\baselineskip
          Let $\dot{\mcalE}$ be a $\mbbB^+_\kappa$-name  such that 
     $\Vdash_{\mbbB^+_\kappa} ``\dot{\mcalE}$ is predense in $\PJmu$ ".
          We show that
        $\Vdash_{\mbbB_\kappa^+} ``\dot{\mcalE}$ is predense in $\PJ$ below $X$".

        Let $\mcalD_{\dot{\mcalE}}=\{(p,\dot{Z})\mid p\Vdash_{\mbbB^+_\kappa} \dot{Z}\in\dot{\mcalE}\}$. $\mcalD_{\dot{\mcalE}}$ is predense in $\mbbB^+_\kappa\ast \PJmu$, so
        $k_0 [\mcalD_{\dot{\mcalE}}]$ is predense in $\mbbB^+_\kappa\ast\PI$ below $(1_{\mbbB^+_\kappa}, X)$ by Claim 4.
        We define 
        \begin{center}
        ${\dot{\mcalE'}}=\{(\dot{Y}, p)\mid (p,\dot{Y})\in k_0[\mcalD_{\dot{\mcalE}}]\}=\{(\dot{Z}\land X,p)\mid p\Vdash_{\mbbB^+_\kappa} \dot{Z}\in \dot{\mcalE}\}$.
        \end{center}
        The predensity of $k_0[\mcalD_{\dot{\mcalE}}]$ implies that $\Vdash_{\mbbB^+_\kappa} ``\dot{\mcalE'}$ is predense in $\PJ$ below $X$".
        Since $\Vdash_{\mbbB^+_\kappa} ``\dot{\mcalE}'=\{Z\land X\mid Z\in \dot{\mcalE}\}"$,
        it follows that $\Vdash_{\mbbB^+_\kappa} ``\dot{\mcalE}$ is predense in $\PJ$ below $X$".
        \end{proof}

        From Lemma \ref{compatible} and \ref{mastercondition}, we obtain the following theorem.

        \begin{theorem}
  Let $\kappa$ be a successor cardinal,
  $\mu$ and $\lambda$ be inaccessible cardinals with $\kappa<\mu<\lambda$,
  $\mcalI$ be a $\kappa$-tower on $\langle \mcalP_\kappa(V_\alpha)\mid \alpha\in [\kappa,\lambda)\rangle$ and
  $\dot{\mcalJ}$ be a $\collkappa$-name of the tower generated by $\mcalI$.
  Suppose that $\mcalI$ and $\mcalI\restrict \mu$ are presaturated,
  and $\mcalI$ factors at $\mu$.
  Then $\Vdash_{\collkappa} `` \dot{\mcalJ}$ factors at $\mu$ ".
 \end{theorem}

\section{Main Theorem}
  In this section we prove the main theorem, Theorem \ref{maintheorem}.
  That proof is based on the proof of Theorem 3.3.15 in Larson \cite{larson}.

  \let\temp\thetheorem
\renewcommand{\thetheorem}{\ref{maintheorem}}

  \begin{theorem}
  Let $\lambda$ be an inaccessible cardinal.
  Suppose that 
  $\mcalI^\kappa_\lambda$ is presaturated for unboundedly many successor cardinals $\kappa<\lambda$
  and 
  there are unboundedly many inaccessible cardinals $\mu<\lambda$ such that $\mcalI^\kappa_\mu$ is presaturated.
  Then every set of reals in $L(\mbbR)$ is $\lambda$-universally Baire.
 \end{theorem}

 \let\thetheorem\temp
\addtocounter{theorem}{-1}

 We use the following lemma.

 \begin{lemma}\label{absolute}
  Suppose the assumption of Theorem \ref{maintheorem}
  holds.
  Let $G$ be a $\mbbQ^{\omega_1}_\lambda$-generic filter over $V$ and $j \colon V\rightarrow M$ be the ultrapower map associated to $G$.
  Suppose that $\kappa<\lambda$ is a successor cardinal in $V$ such that $\mcalI^\kappa_\lambda$ is presaturated,
  and $g$ is a $Coll(\omega,<\kappa)$-generic filter over $V$ with $g\in V[G]$.
  Let $\phi$ be a formula and $\bar{x}$ be a finite sequence of reals in $V[g]$ and strong limit cardinals with the cofinality greater than $\lambda$.
  Then, in $V[G]$,
  \begin{center}
    $L(\mbbR)^M \models \phi(\bar{x}) \Leftrightarrow L(\mbbR)^{V[g]}\models\phi(\bar{x})$.
  \end{center}
 \end{lemma}
 
 \begin{proof}
  Let $K$ be a $Coll(\omega,(2^{2^\lambda})^{V[G]})$-generic filter over $V[G]$.
  In the following proof, we argue in $V[G][K]$.
  By the presaturation of $\mcalI^{\omega_1}_\lambda$ and Corollary \ref{addedreal}, 
  there is a $Coll(\omega,<\lambda)$-generic filter $h\in V[G][K]$ over $V$ such that $\mbbR^{V[G]}=\mbbR^{V[h]}$.

  Let $\mcalJ^\kappa_\lambda$ be the tower generated by $\mcalI^\kappa_\lambda$ in $V[g]$.
  Combining Theorem \ref{imply} and Theorem \ref{presatpres}, it follows from the assumption on $\mcalI^\kappa_\lambda$ that 
  \begin{center}
  $V[g]\models$ ``$\mcalJ^\kappa_\lambda$ is presaturated and 
  there are unboundedly many inaccessible cardinals $\mu<\lambda$ at which $\mcalJ^\kappa_\lambda$ factors". 
  \end{center}
  Let $H$ be a $\mbbP_{\mcalJ^\kappa_\lambda}$-generic filter over $V[g]$ and $k \colon V[g]\rightarrow N$ be the ultrapower map associated to $H$.
\vskip0.5\baselineskip
  \textbf{Claim 1.} 
  $L(\mbbR)^{M}\models \phi(\bar{x}) \Leftrightarrow L(\mbbR)^{N}\models \phi(\bar{x})$.
\vskip0.5\baselineskip
  \textbf{Proof of Claim 1.}
  Using Corollary \ref{addedreal},  we obtain a $Coll(\omega,<\lambda)$-generic filter $h'\in V[G][K]$ over $V[g]$ such that $\mbbR^{V[g][H]}=\mbbR^{V[g][h']}$.
  We note that $L(\mbbR)^M= L(\mbbR)^{V[G]}$ and $L(\mbbR)^N=L(\mbbR)^{V[g][H]}$ by Fact \ref{wellfdd}.
  
  Since $g$ is coded by a real in $V[G]$ and $\mbbR^{V[G]} =\mbbR^{V[h]}$,
  we have $g\in V[h]$.
  So,  we can take a $Coll(\omega, <\lambda)$-generic filter $h^\ast$ over $V[g]$ such that $V[h]=V[g][h^\ast]$.
  By the weakly homogeneity of L\`{e}vy collapses and $\bar{x}\in V[g]$, 
  \begin{center}
  $L(\mbbR)^{V[g][h^\ast]}\models \phi(\bar{x}) \Leftrightarrow  L(\mbbR)^{V[g][h']}\models \phi(\bar{x})$.
  \end{center}
  Since $L(\mbbR)^{V[g][h^*]}=L(\mbbR)^{V[h]}=L(\mbbR)^M$ and 
  $L(\mbbR)^{V[g][h']}=L(\mbbR)^{V[g][H]}=L(\mbbR)^{N}$, we obtain the required equivalence.
  \hfill$\blacksquare$
  \vskip0.5\baselineskip
  
  Note that $k(\mu)=\mu$ for all strong limit cardinal $\mu$ with cofinality greater than $\lambda$.
  The elementarity of $k$ lead to the equivalence that 
  \begin{center}
  $L(\mbbR)^{V[g]}\models \phi(\Bar{x}) \Leftrightarrow L(\mbbR)^N \models \phi(\bar{x})$.
  \end{center}
  Then, by Claim 1, we obtain the required equivalence.
 \end{proof}
Using Lemma \ref{absolute},
we can prove Theorem \ref{maintheorem} by the same argument as Theorem 3.3.15 of Larson \cite{larson}.
For the completeness of this paper, we give the proof.

 \begin{proof}[Proof of Theorem \ref{maintheorem}]
 
  Fix $A\in\mcalP(\mbbR)\cap L(\mbbR)$. 
  Let $E_{>\lambda}$ be the class of strong limit ordinals with cofinality greater than $\lambda$.
  Since $L(\mbbR)$ is the transitive collapse of the Skolem hull of $\mbbR\cup E_{>\lambda}$ in $L(\mbbR)$ 
  and $A$ is invariant under the transitive collapsing,
  there is a formula $\psi$ and  $\bar{x} \in (E_{>\lambda}\cup\mbbR)^{<\omega}$ such that $A$ is the unique set with $L(\mbbR)\models \psi(A,\bar{x})$.
  
  We define $\phi(r,\bar{x})$ to be 
  \begin{center}
  ``$L(\mbbR)\models\exists B (\psi(B,\bar{x})\land r\in B)$".
  \end{center}
  Then $A=\{r\in\mbbR\mid\phi(r,\bar{x})\}$.
  
  Let $\delta>\lambda$ be a limit cardinal with the cofinality greater than $\lambda$ such that $\bar{x}\in V_\delta$, and 
  \begin{center}
      $V[g]\models \phi(r,\bar{x}) 
      \Leftrightarrow
      V[g]_\delta \models \phi(r,\bar{x})$
  \end{center}
  for all successor cardinal $\kappa<\lambda$,
  $\collkappa$-generic filter $g$ over $V$ and $r\in \mbbR^{V[g]}$.
  We can take such $\delta$, because we can take a limit cardinal $\delta>\lambda$ with the cofinality greater than $\lambda$ such that 
  \begin{center}
      $p\Vdash_{\collkappa} \phi(\dot{r},\bar{x})
      \Leftrightarrow
      V_\delta\models ``p\Vdash_{\collkappa} \phi(\dot{r},\bar{x})"$
  \end{center}
  for all $p\in\collkappa$ and $\collkappa$-name $\dot{r}$ of a real
  by the reflection principle.

  Let $v\in[V_\delta]^\omega$. $v$ is said to be \textit{good} if the following are satisfied,
  \begin{itemize}
      \item $\bar{x},\lambda\in v\prec V_\delta$,
      \item Suppose that $\pi : v\rightarrow \bar{v}$ is the transitive collapse, $\kappa\in v\cap \lambda$ is a successor cardinal such that $\mcalI^\kappa_\lambda$ is presaturated,  $Coll(\omega,<\pi(\kappa))$-generic filter $g$ over $\bar{v}$ and $r\in\mbbR\cap \bar{v}[g]$.
      Then 
      \begin{center}
      $\phi(r,\bar{x})$ $\Leftrightarrow$ $\bar{v}[g]\models\phi(r,\pi(\bar{x}))$.
      \end{center}
  \end{itemize}
\vskip0.5\baselineskip
  \textbf{Claim 1.} \{$v\in[V_\delta]^\omega\mid v$ is good\} contains a club set.
\vskip0.5\baselineskip
  \textbf{Proof of Claim 1.}
  We call $v\in [V_\delta]^\omega$ \textit{bad} if $\bar{x},\lambda\in v\prec V_\delta$ but $v$ is not good.
  If $v$ is bad, there are a successor cardinal $\kappa\in v\cap \lambda$ such that $\mcalI^\kappa_\lambda$ is presaturated, a $Coll(\omega,<\pi(\kappa))$-generic filter $g$ over $\bar{v}$ and $r\in\mbbR\cap \bar{v}[g]$ such that 
  \begin{center}
  $\phi(r,\bar{x})\nLeftrightarrow \bar{v}[g]\models\phi(r,\pi(\bar{x}))$
  \end{center}
  where $\pi : v\rightarrow\bar{v}$ is the transitive collapse,
  and we can take a $Coll(\omega,<\kappa)$-name $\dot{\tau}$ such that $\pi(\dot{\tau})_g=r$.
  Then we call such a triple $\langle \kappa, g, \dot{\tau}\rangle$ a witness for the badness of $v$.
  
  Suppose that \{$v\in[V_\delta]^\omega\mid v$ is good\} does not contain a club set.
  Then $X=\{v\in[V_\delta]^\omega \mid v$ is bad\} is stationary.
  By the normality, there are a successor cardinal $\kappa<\lambda$ such that $\mcalI^\kappa_\lambda$ is presaturated and a $Coll(\omega,<\kappa)$-name $\dot{\tau}$ for a real such that 
  \begin{center}
  $Y=\{v\in X \mid \exists g$ ($\langle \kappa, g, \dot{\tau}\rangle$ is a witness of the badness of $v$)\}
  \end{center}
  is stationary.
  Let $Z=\{v\cap V_{\kappa+\omega}\mid v\in Y\}$.
  We note that $Z$ is stationary and $\dot{\tau}\in V_{\kappa+\omega}$.
  If $v\in[V_\delta]^\omega$, $\bar{x}, \lambda\in v\prec V_\delta$ and $v\cap V_{\kappa+\omega}\in Z$ then $v$ is bad and there is $g$ such that $\langle \kappa, g, \dot{\tau}\rangle$ witnesses this badness.
  
  Take a $\mbbQ^{\omega_1}_\lambda$-generic filter $G$ over $V$ such that $Z\in G$ and let $j : V\rightarrow M$ be the elementary embedding associated to $G$.
  In $V$, let $w\prec V_\delta$ be such that $V_{\kappa+\omega}\cup \{\bar{x},\lambda\}\subseteq w$ and $|w|<\lambda$.
  Since crit$(j)=\omega_1,  j(\omega_1)=\lambda$ and ${}^\omega M\subseteq M$ in $V[G]$,
  we have that $j[w]$ is countable in $V[G]$ and $j[w]\in M$.
  Since $Z\in G$,
  \begin{center}
  $j[w]\cap j(V_{\kappa+\omega})=j[V_{\kappa+\omega}]=j[\bigcup Z]\in j(Z)$.
  \end{center}
  By the elementarity of $j$ and the construction of $Z$,
  \begin{center}
  $M\models`` j[w]$ is bad $\land$ $\exists g$($\langle j(\kappa), g, j(\dot{\tau})\rangle$  witnesses the badness of $j[w]$)".
  \end{center}
  We fix $g\in M$ as above.
  Let $\epsilon : j[w]\rightarrow \bar{w}$ be the transitive collapse.
  Note that $\bar{w}$ is also the transitive collapse of $w$.
  Let $\sigma \colon w\rightarrow\bar{w}$ be the transitive collapse.
  Since $V_{\kappa+\omega}\subseteq w$ and $\epsilon\circ j\restrict w =\sigma$, it follows that $\sigma(\kappa)=\epsilon(j(\kappa))=\kappa$ and $\sigma(\dot{\tau})=\epsilon(j(\dot{\tau}))=\dot{\tau}$.
  So $g$ is $Coll(\omega,<\kappa)$-generic filter also over $\bar{w}$.
  Thus the above formula implies that 
  \begin{center}
      $M\models\phi(\epsilon(j(\dot{\tau}))_g, \epsilon(j(\bar{x})))
      \nLeftrightarrow
      \bar{w}[g]\models\phi(\epsilon(j(\dot{\tau}))_{g}, \epsilon(j(\bar{x})))$.
  \end{center}
  Thus 
  \begin{center} $M\models\phi(\dot{\tau}_g,\bar{x}) \nLeftrightarrow$ $\bar{w}[g]\models \phi(\dot{\tau}_g, \sigma(\bar{x}))$.
  \end{center}
  Since $V_{\kappa+\omega}\subseteq w$, the genericity of $g$ over $\bar{w}$ implies the genericity over $V$.
  Because of the elementarity of $\sigma \colon \overline{w}\prec V_\delta$ and the assumption for $\delta$,
  \begin{align*}
      \bar{w}[g]\models \phi(\dot{\tau}_g,\sigma(\bar{x}))
  &\Leftrightarrow V_{\delta}[g]\models \phi(\dot{\tau}_g, \bar{x})\\
  &\Leftrightarrow V[g]\models \phi(\dot{\tau}_g, \bar{x}).
  \end{align*}
  Combining the above arguments, we obtain that
  \begin{center}
  $M \models \phi(\dot{\tau}_g, \bar{x}) \nLeftrightarrow V[g]\models \phi(\dot{\tau}_g, \bar{x})$.
  \end{center}
  This contradicts Lemma \ref{absolute}.
  \hfill$\blacksquare$
\vskip0.5\baselineskip

  From Claim 1, we obtain a function $f : [V_\delta]^{<\omega} \rightarrow V_\delta$ such that if $v\in[V_\delta]^\omega$ and $f[[v]^{<\omega}]\subseteq v$ then $v$ is good.
 
 We define trees $S$ and $T$ on $\omega\times V_\delta\times V_\lambda$ as follows.
 
 $\langle s,t,u \rangle \in S$ if
 \begin{enumerate}
     \item $|s|=|t|=|u|$,
     \item $t(0)=\{\kappa,\dot{\tau}\}$, where $\kappa<\lambda$ is a successor cardinal such that $\mcalI^\kappa_\lambda$ is presaturated and $\dot{\tau}$ is a $Coll(\omega,<\kappa)$-name for a real,
     $t(i+1)=t(i)\cup\{u(i)\}\cup f[[t(i)\cup\{u(i)\}]^{<i}]$,
     \item $u$ is a descending sequence in $Coll(\omega,<\kappa)$ such that
      \begin{enumerate}
          \item if $\mcalD\in t(i)$ is a dense subset of $Coll(\omega,<\kappa)$ then $u(i)\in \mcalD$,
          \item $u(i)\Vdash_{\collkappa} \dot{\tau}(i)=s(i)$ for $i<\omega$,
          \item $u(0)\Vdash_{\collkappa}  \phi(\dot{\tau},\bar{x})$.
      \end{enumerate}      
 \end{enumerate}
 
 $\langle s,t,u\rangle \in T$ if
 $\langle s,t,u\rangle$ satisfies (1),(2),(3)(a) and (b) of the definition of $S$, and the condition
 $u(0)\Vdash_{\collkappa}  \lnot\phi(\dot{\tau},\bar{x})$ instead of (3)(c).
 Note that $t(i)$ is finite for each $i<\omega$.

We see that these trees are witnesses for the universally Baireness of $A$.
\vskip0.5\baselineskip
 \textbf{Claim 2.} $p[S]\subseteq A$ and $p[T]\subseteq \mbbR\setminus A$.
\vskip0.5\baselineskip
 \textbf{Proof of Claim 2.}
  Let $r\in p[S]$.
  Suppose that $\langle r,h,k\rangle \in [S]$ and 
  let $v=\bigcup_{i\in\omega} h(i), h(0)=\{\kappa, \dot{\tau}\}$. 
  Since $v$ is closed under $f$, $v$ is good.
  Let $g$ be the filter of $Coll(\omega,<\pi(\kappa))$ generated by $\{\pi(k(i))\mid i\in\omega\}$, where $\pi : v\rightarrow \bar{v}$ is the transitive collapse map.
  By (3) in the definition of $S$, $g$ is generic over $\bar{v}$, and $r=\pi(\dot{\tau})_{g}$.
  Since $v$ is good, it holds that 
  \begin{center}
  $\phi(r,\bar{x}) \Leftrightarrow \bar{v}[g]\models \phi(r,\bar{x})$.
  \end{center}
  We note that the right-hand side is true from the condition (3)(c).
  So $\phi(r,\bar{x})$ holds, thus $r\in A$.

  In the same way, we can prove that if $r\in p[T]$ then $r\notin A$.
  \hfill$\blacksquare$
\vskip0.5\baselineskip
 \textbf{Claim 3.} 
 Let $\kappa<\lambda$ be a successor cardinal such that $\mcalI^\kappa_\lambda$ is presaturated and $g$ be a $Coll(\omega,<\kappa)$-generic filter over $V$.
 Then $V[g]\models p[S]\cup p[T]=\mbbR$.
\vskip0.5\baselineskip
 \textbf{Proof of Claim 3.}
 Let $r\in \mbbR\cap V[g]$.
 We suppose that $V[g]\models\phi(r,\bar{x})$, and construct $h\in V_{\delta}^\omega$
 and $k\in V_{\lambda}^{\omega}$ such that $\langle r,h,k\rangle\in [S]$.
  We recursively define $h(i)$ and $k(i)$ for $i\in\omega$.
 
 Let $\dot{\tau}$ be a $\collkappa$-name for $r$ and $h(0)=\{\kappa,\dot{\tau}\}$.
 We take $p\in g$ such that 
 \begin{center}
 $p\Vdash_{\collkappa} ``\dot{\tau}(0)=r(0) \land \phi(\dot{\tau},\bar{x})"$,
 \end{center}
 and let $k(0)=p$.
 
 When $h(i)$ and $k(i)$ have been taken, take $h(i+1)$ and $k(i+1)$ as follows.
  \begin{enumerate}
      \item 
         $h(i+1)=h(i)\cup\{k(i)\}\cup f[[h(i)\cup\{k(i)\}]^{<i}]$,
      \item 
        \begin{enumerate}
         \item $k(i+1)\leq_{\collkappa} k(i+1)$ and $k(i)\in g$,
         \item if $\mcalD\in h(i+1)$ is a dense subset of $\collkappa$ then $k(i+1)\in \mcalD$,
         \item $k(i+1)\Vdash_{\collkappa} \dot{\tau}(i+1)=r(i+1)$.
        \end{enumerate}
  \end{enumerate}
  Such $k(i+1)$ exists since $h(i+1)$ is finite, and $g$ is generic.
  
  By the construction, $\langle r,h,k\rangle\in [S]$, so $r\in p[S]$.
  In the same way, we can prove that if $V[g]\models\lnot\phi(r,\bar{x})$
  then $r\in p[T]$.
  \hfill$\blacksquare$
\vskip0.5\baselineskip
  To see that $p[S]=A$ and $p[T]=\mbbR\setminus A$,
  it suffices to prove that $p[S]\cup p[T]=\mbbR$ by Claim 2.
  Let $r\in \mbbR^V$, $\kappa<\lambda$ be a successor cardinal such that $\mcalI^\kappa_\lambda$ is presaturated and $g$ be a $Coll(\omega,<\kappa)$-generic filter over $V$.
  By Claim 3, $V[g]\models ``r\in p[S] \lor r\in p[T]"$.
  We suppose that $V[g]\models r\in p[S]$.
  Let $S_r$ be a tree on $V_\delta\times V_\lambda$ defined by 
  \begin{center}
  $S_r=\bigcup_{n\in\omega}\{\langle t,u\rangle\mid \langle r\restrict n, t,u\rangle\in S\}$.
  \end{center}
  Then $S_r$ is ill-founded in $V[g]$ since $V[g]\models r\in p[S]$.
  By the absoluteness of the ill-foundedness,
  $S_r$ is ill-founded in $V$. That is, $r\in p[S]$.
  In the same way, we can prove that if $V[g]\models r\in p[T]$ then $r\in p[T]$ in $V$.
  The proof is completed.
 \end{proof}

\section{Question}
  We show that the presaturation at $\mu<\lambda$ implies the factorization at $\mu$ for the non-stationary tower $\mcalI^\kappa_\lambda$ in Theorem \ref{imply}.
  How about general $\kappa$-towers?

  \begin{question}
  Let $\kappa$ be a successor cardinal and $\mu, \lambda$ be inaccessible cardinals with $\kappa<\mu<\lambda$.
      Suppose that $\mcalI$ be a $\kappa$-tower on $\langle \mcalP_\kappa(V_\alpha)\mid \alpha\in[\kappa,\lambda)\rangle$.
      Assume that $\mcalI$ and $\mcalI\restrict\mu$ are presaturated.
      Then does $\mcalI$ factor at $\mu$ ?
  \end{question}

  And in this paper we give the sufficient conditions on stationary towers to derive the universally Baireness in $L(\mbbR)$. 
  The large cardinal axioms that imply the universally Baireness in $L(\mbbR)$ have many other consequences, among which the forcing absoluteness is closely related to the universally Baireness.

\begin{question}
    What conditions on stationary towers derive the generic absoluteness of the theory of $L(\mbbR)$ with real parameters?
    In particular, does the assumption of the stationary tower in Theorem \ref{maintheorem} derive that consequence?
\end{question}

\section*{Acknowledgement}
    I would like to express my deepest gratitude to my supervisor Hiroshi Sakai for his detailed teaching and worthwhile discussions on a lot of aspects.
    I would like to thank Daisuke Ikegami and Toshimichi Usuba for their meaningful comments in seminars about this result.

\bibliographystyle{IEEEtran}
\bibliography{stat_Baire_arxiv}

\end{document}